\documentclass[12pt,reqno]{amsart}

\newcommand\version{June 09, 2025}

\usepackage{amsmath,amsfonts,amsthm,amssymb,amsxtra}
\usepackage{hyperref}
\usepackage{bbm} 
\usepackage{bm} 
\usepackage{graphicx}
\usepackage{float}
\usepackage[font=tiny,labelfont=bf,justification=centering]{caption}
\makeindex 

\newtheorem{theorem}{Theorem}[section]
\newtheorem{proposition}[theorem]{Proposition}
\newtheorem{lemma}[theorem]{Lemma}
\newtheorem{corollary}[theorem]{Corollary}

\theoremstyle{definition}

\newtheorem{definition}[theorem]{Definition}

\theoremstyle{remark}

\newtheorem{remark}[theorem]{Remark}

\numberwithin{equation}{section}

\newcommand{\C}{\mathbb{C}}

\newcommand{\D}{\mathcal{D}}

\renewcommand{\epsilon}{\varepsilon}

\newcommand{\F}{\mathcal{F}}

\newcommand{\loc}{{\rm loc}}

\newcommand{\N}{\mathbb{N}}

\renewcommand{\phi}{\varphi}
\newcommand{\R}{\mathbb{R}}

\newcommand{\Z}{\mathbb{Z}}

\DeclareMathOperator{\spann}{span}

\DeclareMathOperator{\supp}{supp}
\DeclareMathOperator{\sgn}{sgn}

\def\bs{\mathbb{S}}

\def\ca{\mathcal{A}}

\def\ce{\mathcal{E}}

\def\ci{\mathcal{I}}
\def\cj{\mathcal{J}}

\def\cl{\mathcal{L}}
\def\fl{\mathfrak{L}}

\def\cq{\mathcal{Q}}

\def\Rd{{\mathbb{R}^d}}

\newcommand{\me}[1]{\mathrm{e}^{#1}}
\newcommand{\one}{\mathbf{1}}

\begin{document}

\title[Heat kernel bounds --- \version]{Heat kernel bounds for the fractional Laplacian with Hardy potential in angular momentum channels}

\author[K. Bogdan]{Krzysztof Bogdan}
\address[Krzysztof Bogdan]{Department of Pure and Applied Mathematics, Wroc\l aw University of Science and Technology, Hoene-Wro\'nskiego 13C, 50-376 Wroc\l aw, Poland}
\email{krzysztof.bogdan@pwr.edu.pl}
\thanks{The first-named author was  supported the Opus grant 2023/51/B/ST1/02209 of the National Science Center, Poland.}

\author[K. Merz]{Konstantin Merz}
\address[Konstantin Merz]{Institut f\"ur Analysis und Algebra, Technische Universit\"at Braunschweig, Universit\"atsplatz 2, 38106 Braun\-schweig, Germany, and Institut f\"ur Partielle Differentialgleichungen, Technische Universit\"at Braunschweig, Uni\-ver\-si\-t\"ats\-platz 2, 38106 Braun\-schweig, Germany}
\email{k.merz@tu-bs.de}

% ...

\date{\version}

\begin{abstract}
  Motivated by the study of relativistic atoms, we prove sharp heat kernel bounds for the Hardy operator $(-\Delta)^{\alpha/2}-\kappa|x|^{-\alpha}$ acting on functions of the form $u(|x|) |x|^{\ell} Y_{\ell,m}(x/|x|)$ in $L^2(\R^d)$, when $\alpha\in(0,2]\cap(0,d+2\ell)$.
\end{abstract}


\maketitle
\vspace*{-2em}
\tableofcontents

\section{Introduction and main result}
\label{s:introduction}

We consider the Laplacian $\Delta$ and the fractional Laplacian $(-\Delta)^{\alpha/2}$ with the added Hardy potential $-\kappa/|x|^\alpha$ in $L^2(\R^d)$. Here, $d \in \N := {1,2,\ldots}$, $\alpha \in (0, 2]$, with $\alpha < d$, and $\kappa \in \R$ is called the coupling constant. Note that $|x|^{-\alpha}$ is locally integrable on $\Rd$.
We call the Hardy potential attractive (mass-creating) if $\kappa>0$ and repulsive (mass-killing) if $\kappa<0$. Note that the operators $(-\Delta)^{\alpha/2}$ and the multiplication by $-\kappa|x|^{-\alpha}$ are homogeneous of the same degree (with respect to dilations of functions). 
We will study their sum 
in a connection to the following inequality,
\begin{align}
  \label{eq:hardy}
  \begin{split}
    & \left\| (-\Delta)^{\frac\alpha4} f \right\|_{L^2(\R^d)}^2
    \geq \kappa_{\rm c}^{(\alpha)}(d) \left\| |x|^{-\frac\alpha2} f \right\|_{L^2(\R^d)}^2, \quad f\in C_c^\infty(\R^d).
  \end{split}
\end{align}
known as the Hardy--Kato--Herbst inequality.
Here, again, $\alpha\in(0,2]\cap(0,d)$ and the inequality is sharp if we take
\begin{align}
  \kappa_{\rm c}^{(\alpha)}(d) := 2^\alpha\, \frac{\Gamma\left(\frac{d+\alpha}{4}\right)^2}{\Gamma\left(\frac{d-\alpha}{4}\right)^2},
\end{align}
see \cite{Hardy1919,Hardy1920,Kato1966,Herbst1977} and
\cite{Kovalenkoetal1981,Yafaev1999,Franketal2008H,FrankSeiringer2008}.
We simply call~\eqref{eq:hardy} the Hardy inequality. 
Thus, the quadratic form associated with $(-\Delta)^{\alpha/2} - \kappa |x|^{-\alpha}$ on $C_c^\infty(\R^d)$ is bounded from below if and only if $\kappa \leq \kappa_{\rm c}^{(\alpha)}(d)$, in which case, the form is non-negative. Then, by a theorem of Friedrichs, there is a corresponding non-negative, self-adjoint operator on $L^2(\Rd)$, which we call the Hardy operator and denote by $(-\Delta)^{\alpha/2} - \kappa |x|^{-\alpha}$.

Hardy and other homogeneous operators frequently appear as effective operators or as scaling limits of more complicated operators. In our recent works \cite{BogdanMerz2024, BogdanMerz2025, Bogdanetal2024} and in the present paper, we study Hardy operators by exploiting their spherical symmetry, leading to a direct sum decomposition into summands acting on radial multiples of spherical harmonics with fixed degree. Our main result, Theorem~\ref{mainresult} below, provides sharp (matching) upper and lower bounds for the heat kernels of these direct summands: In doing so, we also relax the condition $\kappa\leq\kappa_{\rm c}^{(\alpha)}(d)$.

To state the result precisely, consider the $(d-1)$-dimensional unit sphere $\bs^{d-1}\subset \Rd$ and the orthonormal basis of (hyper)spherical harmonics $\{Y_{\ell,m}\}_{\ell\in L_d,m\in M_{d,\ell}}$\index{$M_{d,\ell}$} in $L^2(\bs^{d-1})$. Here $L_1=\{0,1\}$, $L_d=\N_0$ for $d\geq2$, and $M_{d,\ell}\subseteq\Z^{d-2}$ is a finite index set\footnote{An explicit description of $M_{d,\ell}$ is not important here, but can be found in \cite{BogdanMerz2024}.} for each $\ell\in L_d$. In the following, we write $M_\ell$\index{$M_\ell$} instead of $M_{d,\ell}$.
For $\ell\in L_d$, $m\in M_\ell$, and an almost everywhere defined Borel function $u:\R_+\to \C$, we let
\begin{align}\index{$[u]_{\ell,m}$}
  \label{eq:defeqclassellm}
  [u]_{\ell,m}(x) := u(|x|) |x|^{\ell} Y_{\ell,m}(\omega_x),
  \quad \text{ for a.e. } \ x\in\R^d\setminus\{0\}, 
\end{align}
\noindent
where $\omega_x=x/|x|$ for $x\in\R^d\setminus\{0\}$ and $\R_+:=(0,\infty)$.\index{$\omega_x$}
By using polar coordinates, 
\begin{align}
  \label{eq:isometry}
  \|[u]_{\ell,m}\|_{L^2(\R^d)} = \|u\|_{L^2(\R_+,r^{d+2\ell-1}dr)}.
\end{align}
Let
\begin{equation}
  \label{eq:defvelll2}
  V_{\ell,m} := \{ [u]_{\ell,m}:\, u\in L^2(\R_+,r^{d+2\ell-1}dr)\},
  \quad \ell\in L_d, \; m\in M_\ell,
\end{equation}\index{$V_{\ell,m}$}and 
\begin{equation}
  \label{eq:defvelll2t}
  V_{\ell} := \spann\{ [u]_{\ell,m}:\, u\in L^2(\R_+,r^{d+2\ell-1}dr), \; m\in M_\ell\},
  \quad \ell\in L_d.
\end{equation}\index{$V_\ell$}
We call each $V_\ell$ \emph{angular momentum channel}, and say that functions in $V_{\ell}$ have \emph{angular momentum $\ell$}. For $\ell\in L_d$, we introduce the \emph{effective dimension} $d_\ell:=d+2\ell\in\N$.\index{$d_\ell$} As the spherical harmonics form an orthonormal basis of $L^2(\bs^{d-1})$, we get
\begin{align}
  \label{eq:directsumvell}
  V_\ell= \bigoplus_{m\in M_\ell}  V_{\ell,m}
  \qquad \text{and} \qquad
  L^2(\R^d) = \bigoplus_{\ell\in L_d}V_{\ell},
\end{align}
see, e.g., \cite{BogdanMerz2024} or Kalf \cite{Kalf1995} for an extensive review about the expansion of functions in terms of spherical harmonics.
Thus, in view of the angular momentum decomposition in Lemma~\ref{fourierbesselcor} below, we have the direct sum decomposition
\begin{align}
  \label{eq:angularmomentumdecomposition}
  (-\Delta)^{\alpha/2} - \kappa/|x|^\alpha
  = \bigoplus_{\ell\in L_d} (U_\ell^*\cl_{\kappa,\ell} U_\ell)\otimes\one_{L^2(\bs^{d-1})}.
\end{align}
Here $\one_{L^2(\bs^{d-1})}$ denotes the identity operator on $L^2(\bs^{d-1})$ and $U_\ell:L^2(\R_+,r^{d-1}dr)\to L^2(\R_+,r^{d_\ell-1}dr)$ is the unitary operator defined as $(U_\ell f)(r)=r^{-\ell}f(r)$. The object of main interest in this paper is $\cl_{\kappa,\ell}$, which is that operator in $L^2(\R_+,r^{d_\ell-1}dr)$ arising in
\begin{align}
  \label{eq:lkappaell}
  \langle[u]_{\ell,m},((-\Delta)^{\alpha/2}-\kappa/|x|^\alpha)[u]_{\ell,m}\rangle_{L^2(\R^d)} = \langle u,\cl_{\kappa,\ell}u\rangle_{L^2(\R_+,r^{d_\ell-1}dr)}, 
\end{align}
where $u\in C_c^\infty(\R_+)$, see also Lemma~\ref{fourierbessel}. Formally, we have
\begin{align}
  \label{eq:deflkappaell}
  \cl_{\kappa,\ell} = \left(-\frac{d^2}{dr^2}-\frac{d_\ell-1}{r}\frac{d}{dr}\right)^{\alpha/2} - \frac{\kappa}{r^\alpha} \quad \text{in} \ L^2(\R_+,r^{d_\ell-1}dr).
\end{align}
By Hardy's inequality \eqref{eq:hardy}, the operators $\cl_{\kappa,\ell}$ are non-negative self-adjoint operators in $L^2(\R_+,r^{d_\ell-1}dr)$ whenever $\kappa\leq\kappa_{\rm c}^{(\alpha)}(d)$, but for $\ell\geq1$, $\cl_{\kappa,\ell}$ can be realized as a non-negative self-adjoint operator even for some $\kappa>\kappa_{\rm c}^{(\alpha)}(d)$. In fact, for $\alpha\in(0,2]\cap(0,d_\ell)$, we let
\begin{align}
  \label{eq:defsigmagammal}
  \Phi_{d_\ell}^{(\alpha)}(\eta) 
  := \frac{2^{\alpha} \Gamma\left(\frac{\eta+\alpha}{2}\right) \Gamma \left(\frac{d_\ell-\eta}{2}\right)}{\Gamma \left(\frac{\eta}{2}\right) \Gamma \left(\frac{d_\ell-\eta-\alpha}{2}\right)},
  \quad \eta \in (-S,d_\ell-\alpha+S),
\end{align}
to parameterize $\kappa$. Here, $S:=\alpha$ if $\alpha<2$, but $S:=\infty$ if $\alpha=2$. Note that
\begin{align}
  \Phi_{d_\ell}^{(2)}(\eta) = \eta(d_\ell-\eta-2), \quad \eta\in \R.
\end{align}
The maximal value of $\Phi_{d_\ell}^{(\alpha)}(\eta)$ is 
\begin{align}\index{$\kappa_{\rm c}^{(\alpha)}(d_\ell)$}
  \label{eq:defgammacell}
  \kappa_{\rm c}^{(\alpha)}(d_\ell) = \Phi_{d_\ell}^{(\alpha)}\left(\frac{d_\ell-\alpha}{2}\right)
  = \frac{2^{\alpha} \Gamma \left(\frac{1}{4} (d_\ell+\alpha)\right)^2}{\Gamma \left(\frac{1}{4} (d_\ell-\alpha)\right)^2},
\end{align}
which simplifies for $\alpha=2$ to $\kappa_{\rm c}^{(2)}(d_\ell)=(d_\ell-2)^2/4$. Thus, for each $\kappa\leq\kappa_{\rm c}^{(\alpha)}(d_\ell)$, there is a unique $\eta \leq (d_\ell-\alpha)/2$ such that $\kappa=\Phi_{d_\ell}^{(\alpha)}(\eta)$. For instance, $0=\Phi_{d_\ell}^{(\alpha)}(0)$. 

By the sharp Hardy inequalities---which are well known for $\alpha=2$ and are first established in Le Yaouanc, Oliver, and Raynal \cite{LeYaouancetal1997} for $d=3$, $\ell\in\N_0$, and $\alpha=1$, and in Yafaev \cite{Yafaev1999} for all $d\in\N$, $\ell\in L_d$, and $\alpha\in(0,2\wedge d_\ell)$---we have
\begin{align}
  \label{eq:hardylinformal}
  \cl_{\kappa,\ell} \geq0 \quad \text{if and only if}\ \kappa\leq\kappa_{\rm c}^{(\alpha)}(d_\ell)
\end{align}
in the sense of quadratic forms on $C_c^\infty(\R_+)$. In fact, \eqref{eq:hardylinformal} may be viewed as a consequence of the ground state representation \cite{BogdanMerz2024}, which we recall in \eqref{eq:hardyremainderagaintransformedlimit} and \eqref{eq:hardyopformradial} below. 
In particular, for $\alpha<2$, the ground state representation says that for all $u\in L^2(\R_+,r^{d_\ell-1}dr)$ the quadratic form $\|(-\Delta)^{\alpha/4}[u]_{\ell,m}\|_{L^2(\R^d)}^2-\Phi_{d_\ell}^{(\alpha)}(\eta)\||x|^{-\alpha/2}[u]_{\ell,m}\|_{L^2(\R^d)}^2$ associated with $\cl_{\kappa,\ell}$ equals an explicit nonnegative integral form, denoted by $\ci_{(d_\ell-1)/2,\eta}$ and defined in \eqref{eq:defizetaeta} below; see \eqref{eq:fractlaplacesingularintegral}, \eqref{eq:hardyforml}, and \eqref{eq:hardyopformradial}. In Theorem~\ref{relationhardyformheatkernel}, we show that $\ci_{(d_\ell-1)/2,\eta}$ on its maximal domain is nonnegative, closed, and symmetric. Thus, by a theorem of Friedrichs (see, e.g., \cite[Theorem~VI.2.1]{Kato1966} or \cite[Theorem~1.3.1]{Fukushimaetal2011}), there is a corresponding self-adjoint operator, which we denote by $\cl_{\kappa,\ell}$ (in abuse of notation).
Furthermore, in \cite{BogdanMerz2024}, we show that
\begin{align}
  \label{eq:defh}
  h(r) := r^{-\eta}  
\end{align}
is a so-called \textit{generalized ground state} of $\cl_{\Phi_{d_\ell}^{(\alpha)}(\eta),\ell}$. This means that $\cl_{\Phi_{d_\ell}^{(\alpha)}(\eta),\ell}h(r)=0$ pointwise a.e., but $h(r)$ does not belong to $L^2(\R_+,r^{d_\ell-1}dr)$. 
When $\alpha=2$, we write
\begin{align}
  \cl_{\Phi_{d_\ell}^{(2)}(\eta),\ell} = \fl_{(d_\ell-1)/2} - \frac{\Phi_{d_\ell}^{(2)}(\eta)}{r^2}
\end{align}
with the Bessel operator in Liouville form
\begin{align}
  \label{eq:deflzeta}
  \fl_\zeta := -\frac{d^2}{dr^2} - \frac{2\zeta}{r}\frac{d}{dr}
  \quad \text{in} \ L^2(\R_+,r^{2\zeta}dr).
\end{align}\index{$\fl_\zeta$}More precisely, $\fl_\zeta$ is defined as the Friedrichs extension of the corresponding quadratic form on $C_c^\infty(\R_+)$ when $\zeta\in[1/2,\infty)$, while $\fl_\zeta$ is defined as the Krein extension of the corresponding quadratic form on $C_c^\infty(\R_+)$ when $\zeta\in(-1/2,1/2]$. In both cases, $\fl_\zeta$ is non-negative and self-adjoint. We refer, e.g., to \cite[Theorem~4.22]{Bruneauetal2011}, \cite[Theorem~8.4]{DerezinskiGeorgescu2021}, or \cite[Section~4]{Metafuneetal2018} for further details and references regarding the spectral theory of $\fl_\zeta$, in particular for the operator and quadratic form domains and (form) cores.
By a change of variables, the quadratic forms associated to $\cl_{\kappa,\ell}$ in $L^2(\R_+,r^{d_\ell-1})$ and $\fl_{(d_\ell-1)/2-\eta}$ in $L^2(\R_+,r^{d_\ell-1-2\eta}dr)$ coincide on $C_c^\infty(\R_+)$, see, e.g., \cite{Metafuneetal2018} or \cite{BogdanMerz2024}. Hence, for $\alpha=2$, we define $\cl_{\Phi_{d_\ell}^{(2)}(\eta),\ell}$ as that self-adjoint operator in $L^2(\R_+,r^{d_\ell-1}dr)$ which is unitarily equivalent to $\fl_{(d_\ell-1)/2-\eta}$ in $L^2(\R_+,r^{d_\ell-1-2\eta}dr)$. Note that $(d_\ell-1)/2-\eta\geq1/2$ for $\eta\leq(d_\ell-2)/2$.


\smallskip
We now state our main result, the sharp bounds for heat kernel $\exp\left(-t\cl_{\Phi_{d_\ell}^{(\alpha)}(\eta),\ell}\right)$ of $\cl_{\Phi_{d_\ell}^{(\alpha)}(\eta),\ell}$, which is defined by functional calculus. These bounds give an approximate factorization of the heat kernel into a product of the heat kernel of $\cl_{0,\ell}$ times singular weights related to the generalized ground state $r^{-\eta}$.
\begin{theorem}
  \label{mainresult}
  Let $d\in\N$, $\ell\in L_d$, $\alpha\in(0,2]\cap(0,d_\ell)$, and $\eta\in(-S,(d_\ell-\alpha)/2]$.
  Then, for $\alpha\in(0,2)$, uniformly in $r,s,t>0$,
  \begin{align}
    \label{eq:mainresultalpha}
    \begin{split}
      & \exp\left(-t\cl_{\Phi_{d_\ell}^{(\alpha)}(\eta),\ell}\right)(r,s)
        \sim_{d,\ell,\alpha,\eta} \left(1\wedge\frac{r}{t^{1/\alpha}}\right)^{-\eta} \left(1\wedge\frac{s}{t^{1/\alpha}}\right)^{-\eta} \cdot \exp\left(-t\cl_{0,\ell}\right)(r,s), \\
    \end{split}
  \end{align}
  while for $\alpha=2$, 
  \begin{align}
    \label{eq:mainresult2}
    \begin{split}
      & \exp\left(-t\cl_{\Phi_{d_\ell}^{(2)}(\eta),\ell}\right)(r,s)
        \asymp_{d,\ell,\eta} \left(1\wedge\frac{r}{t^{1/2}}\right)^{-\eta} \left(1\wedge\frac{s}{t^{1/2}}\right)^{-\eta} \cdot \exp\left(-t\cl_{0,\ell}\right)(r,s). \\
     \end{split}
  \end{align}
  Furthermore, for $\alpha<2$, we have
  \begin{align}
    \label{eq:mainresultfreealpha}
    \begin{split}
       \exp\left(-t\cl_{0,\ell}\right)(r,s)
       \sim_{d,\ell,\alpha} \frac{t}{|r-s|^{1+\alpha}(r+s)^{d_\ell-1}+t^{\frac{1+\alpha}{\alpha}}(t^{\frac1\alpha}+r+s)^{d_\ell-1}} 
    \end{split}
  \end{align}
  and, for $\alpha=2$,
  \begin{align}
    \label{eq:mainresultfree2}
    \begin{split}
       \exp\left(-t\cl_{0,\ell}\right)(r,s)
       \asymp_{d,\ell} t^{-\frac12}\frac{\exp\left(-\frac{(r-s)^2}{ct}\right)}{(rs+t)^{(d_\ell-1)/2}}.
    \end{split}
  \end{align}
  Moreover, $\exp\left(-t\cl_{\Phi_{d_\ell}^{(\alpha)}(\eta),\ell}\right)(r,s)$ is jointly continuous as a function of $r,s,t>0$.
\end{theorem}

The main point of Theorem~\ref{mainresult} is the connection between the semigroup constructed in \cite{Bogdanetal2024} (which we recall in Subsection~\ref{ss:defperturbedsemigroup} below) and the heat kernel $\me{-t\cl_{\kappa,\ell}}$ of $\cl_{\kappa,\ell}$ and the corresponding generating quadratic forms (which we recall in Subsection~\ref{s:defquadformsemigroup} below).

Let us explain the notation used in Theorem~\ref{mainresult}.
We write $A\lesssim B$ for functions $A,B\geq0$ to indicate that there is a \textit{constant} $c\in (0,\infty)$ such that $A\leq c B$. If $c$ depends on $\tau$, we may write $c_\tau$ instead of $c$ or $A\lesssim_\tau\! B$, but the dependence on $d$, $\ell$, and $\alpha$ may also be ignored in the notation. Note that the values of constants may change from place to place. The notation $A\sim_\tau B$ means $A\lesssim_\tau\! B \lesssim_\tau\! A$ and we then say that $A$ is \emph{equivalent} or comparable to $B$.\index{$\lesssim$}\index{$\sim$}
The notation $\asymp$\index{$\asymp$} means the same as $\sim$, but the constants in the \textit{exponential factors} in an estimate are allowed to be different in the upper and the lower bounds; see \eqref{eq:mainresult2} and \eqref{eq:mainresultfree2}. As usual, $A\wedge B:=\min\{A,B\}$ and $A\vee B:=\max\{A,B\}$. \index{$\wedge$}\index{$\vee$} 

\begin{remark}
  For future applications, it would be important to describe the dependence of constants in our estimates on the angular momentum $\ell$, for instance for the angular momentum synthesis \eqref{eq:angularmomentumdecomposition} and Lemma~\ref{fourierbesselcor}. However, tracking the dependence of the estimates in Theorem~\ref{mainresult} on $\ell$ is rather involved and is not attempted here.
\end{remark}

\subsection*{Implications of Theorem~\ref{mainresult}}

Heat kernel estimates for Schr\"odinger operators provide a powerful tool kit to approach a wide range of problems, from pure mathematics, such as partial differential equations or spectral theory, to applied fields such as quantum physics or statistical mechanics.
Our works \cite{BogdanMerz2024,BogdanMerz2025,Bogdanetal2024}, and the present one are mainly motivated by the ambition to establish functional inequalities relating Hardy operators to fractional Laplacians, and to investigate the electron distribution in large, relativistically described atoms. Indeed, one of the mathematically simplest models describing the energy and time evolution of $Z$ relativistic electrons interacting with each other and with a static nucleus of electric charge $Z$ is given by the many-particle Chandrasekhar (or pseudorelativistic) operator
\begin{align}
  \label{eq:manyparticlechandrasekhar}
  \begin{split}
    H_{Z,c}
    & := \sum_{\nu=1}^Z \underbrace{\one_{L^2(\R^3)}\otimes\cdots\otimes\one_{L^2(\R^3)}}_{\nu-1\text{ times}} \otimes C_{Z,c} \otimes \underbrace{\one_{L^2(\R^3)}\otimes\cdots\otimes \one_{L^2(\R^3)}}_{Z-\nu\text{ times}} + \sum_{1\leq\nu<\mu\leq Z}\frac{1}{|x_\nu-x_\mu|}
  \end{split}
\end{align}
defined in the totally antisymmetric tensor product of $Z$ copies of $L^2(\R^3)$. Here,
\begin{align}
  \label{eq:chandrasekhar}
  C_{Z,c} := \sqrt{c^4-c^2\Delta}-c^2-\frac{Z}{|x_\nu|} \quad \text{in} \ L^2(\R^3)
\end{align}
describes a single electron interacting with the nucleus and $c>0$ denotes the speed of light. The operator $C_{Z,c}$ has a natural length scale given by $c^{-1}$ and is unitarily equivalent to $c^2C_\kappa$ with $C_\kappa:=C_{\kappa,1}$ and $\kappa=Z/c$. Since $C_\kappa$ is a bounded perturbation of $\sqrt{-\Delta}-\kappa/|x|$, by the Hardy inequality \eqref{eq:hardy}, $C_{Z,c}$ and $H_{Z,c}$ are bounded from below if and only if $\kappa\leq\kappa_{\rm c}^{(1)}(3)=2/\pi$. Moreover, according to \cite{Lewisetal1997}, the lowest spectral point of $H_{Z,c}$ is an eigenvalue, called the ground state energy. This eigenvalue may be degenerate, but to keep the discussion simple, let us suppose it is simple and let $\psi$ denote a corresponding eigenfunction, called the ground state. In \cite{Franketal2020P}\footnote{See also \cite{Franketal2023} for a shorter proof and \cite{Franketal2023T} for a review.}, we study the probability density of finding one of the electrons at $x$ given by
\begin{align}
  \rho_Z(x) := Z\int_{\R^{3(Z-1)}}|\psi(x,x_2,...,x_Z)|^2\,dx_2\ldots\,dx_Z
\end{align}
and called the ground state density.
We show that $c^{-3}\rho_Z(x/c)$ converges in the asymptotic limit $Z,c\to\infty$, with fixed $\kappa<2/\pi$, to $\rho^H(x):=\sum_{n\geq1}|\Phi_n(x)|^2$, where $\Phi_n$ are the infinitely many eigenfunctions of $C_\kappa$. Thus, the electron distribution of a large atom close to the nucleus is well described by the electron distribution of an infinite atom without electron-electron interactions. Moreover, we prove that $\rho^H(x)\lesssim|x|^{-2\eta}$, where $\eta$ is defined by $\kappa=\Phi_{d_0}(\eta)$. We note in passing that the corresponding results for the physically more relevant Coulomb--Dirac operator are proved in \cite{MerzSiedentop2022}.
A paramount role in proving both the convergence result and the upper bound for $\rho^H(x)$ is played by the recent heat kernel bounds for $\sqrt{-\Delta}-\kappa/|x|$ with $\kappa>0$, as established in the pioneering paper \cite{Bogdanetal2019} by Bogdan, Grzywny, Jakubowski, and Pilarczyk; see also Jakubowski and Wang \cite{JakubowskiWang2020} and Cho, Kim, Song, and Vondra\v{c}ek \cite{Choetal2020} for the case $\kappa<0$. In particular, the proof of the convergence relies on the validity of certain relative Schatten bounds with respect to powers of $C_\kappa$ greater than one. Establishing such bounds is non-trivial for \(\kappa \neq 0\) because the Hardy inequality can no longer be used to replace \(C_\kappa\) with the more manageable Fourier multiplier \(C_0\); this is due to the fact that powers greater than one are not operator monotone.

However, we overcome this obstruction in \cite{Franketal2021}. There, we use the heat kernel bounds for \( (-\Delta)^{\alpha/2} - \frac{\kappa}{|x|^\alpha} \) and prove that the \( L^2(\mathbb{R}^d) \)-Sobolev spaces generated by powers of \( (-\Delta)^{\alpha/2} - \frac{\kappa}{|x|^\alpha} \) coincide with the ordinary homogeneous Sobolev spaces when \(\kappa\) is not too large in a quantitative sense. For the case \(\alpha = 2\), see the influential paper \cite{Killipetal2018} by Killip, Miao, Visan, Zhang, and Zheng; for corresponding results involving \( L^p \)-Sobolev spaces with \( p \neq 2 \), we refer to \cite{Merz2021,BuiDAncona2023}; and for similar investigations involving the fractional Laplacian on the half-space \( \mathbb{R}_+^d \) with Hardy potential \( \frac{\kappa}{x_d^\alpha} \), depending on the distance to the bounding half-plane, see \cite{FrankMerz2023,BuiMerz2023}.

The present work allows us to refine the above results on the electron distribution of large atoms. Let us sketch the argument. Due to the spherical symmetry, the eigenfunctions of $C_\kappa$ are of the form $|x|^{-\ell} \phi_{n,\ell}(|x|)Y_{\ell,m}(x/|x|)$ with $\phi_{n,\ell} \in L^2(\R_+,r^{d_\ell-1}dr)$, $d=3$. In \cite{Franketal2020P}, we also prove the more precise statement that the ground state density of electrons with prescribed angular momentum $\ell\in\N_0$, given by
\begin{align}
  \varrho_{Z,\ell}(r) := \frac{Z}{(2\ell+1) r^{2\ell}} \sum_{m=-\ell}^\ell \int_{\R^{3(Z-1)}}\left|\int_{\bs^2}d\omega\, \overline{Y_{\ell,m}(\omega)} \psi(r\omega,x_2,...,x_Z)\right|^2dx_2\ldots dx_Z,
\end{align}
converges to $\varrho_\ell^H(r):=\sum_{n\geq1}|\phi_{n,\ell}(r)|^2$ as $Z,c\to\infty$ with $\kappa<2/\pi$.
Thus, as an application of Theorem~\ref{mainresult}, in \cite{FrankMerz2025} (see also \cite{Merz2024H}), we prove upper bounds for $\varrho_\ell^H(r)$ and show that $\varrho_\ell^H(r)\lesssim r^{-2\eta}$ for $r\lesssim1$; in fact, we do so for $\varrho_\ell^H$ defined in terms of the eigenfunctions of $(-\Delta+1)^{\alpha/2}-1-\kappa/|x|^\alpha$ in $L^2(\R^d)$. In addition, in \cite{FrankMerz2025}, we will also use the results of Theorem~\ref{mainresult} to prove upper bounds for $\varrho_\ell^H(r)$ and $\sum_{\ell\in L_d}|M_\ell| \cdot \varrho_\ell^H(r)$ for $r\gtrsim1$.
Moreover, Theorem~\ref{mainresult} can be used to refine the results on the equivalence of Sobolev norms generated by $(-\Delta)^{\alpha/2}-\kappa/|x|^\alpha$ for different values of $\kappa$, and to prove the convergence of the ground state density $\rho_Z$ in fixed angular momentum channels as $Z,c,\to\infty$ with fixed $\kappa=Z/c$ for a wider range of admissible coupling constants $\kappa$.

Besides these concrete future applications, we also believe that the present work serves as a solid starting point for a more detailed analysis of the considered heat kernels regarding, e.g., their asymptotics and smoothness.

\subsection*{Organization and notation}

The rest of the paper is dedicated to proving Theorem~\ref{mainresult} and is structured as follows.
In Section~\ref{s:framework}, we introduce our framework.
In Section~\ref{s:pointwise}, we review subordinated Bessel heat kernels and their Hardy perturbations. In particular, in Theorem~\ref{mainresultgen}, we recall the instrumental bounds for these Hardy perturbations proved in our work \cite{Bogdanetal2024} with Tomasz Jakubowski.
In Section~\ref{s:quadraticforms}, we connect the quadratic forms associated with $\cl_{\kappa,\ell}$ with those generated by the Hardy perturbations of the subordinated Bessel heat kernels. This is precisely stated in our main technical result, Theorem~\ref{relationhardyformheatkernel}, which allows us to conclude the proof of Theorem~\ref{mainresult} in the final Section~\ref{s:proof}.

\medskip
For $z\in\C\setminus(-\infty,0]$ we denote the ordinary and modified Bessel functions of the first kind of order $\nu\in\C$ by $J_\nu(z)$ and $I_\nu(z)$, respectively \cite[(10.2.2), (10.25.2)]{NIST:DLMF}.\index{$J_\nu$}\index{$I_\nu$}
We write $_2\tilde F_1(a,b;c;z) =\, _2F_1(a,b;c;z)/\Gamma(c)$, with $a,b,c\in\C$ and $z\in\{w\in\C:\,|w|<1\}$, for the regularized hypergeometric function \cite[(15.2.1)]{NIST:DLMF} and $_2F_1(a,b;c;z)$ for the (usual) hypergeometric function \cite[(15.2.2)]{NIST:DLMF} when $c\notin\{0,-1,-2,...\}$.\index{$_2F_1(a,b;c;z)$}\index{$_2\tilde F_1(a,b;c;z)$}
We introduce further notation as we proceed---the notation is summarized in the \hyperref[index]{\indexname} at the end of this paper.

\subsection*{Acknowledgments}
We are grateful to Volker Bach, Rupert Frank, and Haruya Mizutani for helpful discussions.
K.M.~thanks Haruya Mizutani for his hospitality at Osaka University, where part of this research was carried out.

\section{The framework}
\label{s:framework}
\subsection{Preliminaries}
\label{ss:preliminaries}

For $d\in\N$ and $\alpha\in(0,2)$, we consider the quadratic form
\begin{align}
  \label{eq:fractlaplacesingularintegral}
  \ce^{(\alpha)}[f]
  :=\frac12\int_{\R^d}\!\int_{\R^d} 
  |f(x)-f(y)|^2\nu(x,y)
  \,dy\,dx, \quad f\in L^2(\R^d),
\end{align}\index{$\ce^{(\alpha)}$}associated to the fractional Laplacian $(-\Delta)^{\alpha/2}$, where
\begin{equation*}
  \nu(x,y) := \nu(x-y),\quad x,y\in \R^d,
\end{equation*}\index{$\nu(x,y)$}with
\begin{align}
  \label{eq:defaalphad}
  \nu(z):=\mathcal{A}_{d,-\alpha}|z|^{-d-\alpha}
  \quad 
  \text{and}
  \quad
  \mathcal{A}_{d,-\alpha}
  := \frac{2^{\alpha}\Gamma\big((d+\alpha)/2\big)}{\pi^{d/2}|\Gamma(-\alpha/2)|}.
\end{align}\index{$\ca_{d,-\alpha}$}There is an equivalent representation of $\ce^{(\alpha)}$ using semigroups. With
\begin{align}
  \label{eq:defusualheatkernel}
  P^{(\alpha)}(t,x,y) := \frac{1}{(2\pi)^d}\int_{\R^d}d\xi\, \me{-i\xi\cdot(x-y)-t|\xi|^\alpha} , \quad x,y\in\R^d,\,t>0,
\end{align}\index{$P^{(\alpha)}(t,x,y)$}we have
\begin{align}
  \label{eq:fractlaplaceheatkernel}
  \ce^{(\alpha)}[f]
  = \lim_{t\to0}\frac1t\langle f,(1-P^{(\alpha)}(t,\cdot,\cdot))f\rangle_{L^2(\R^d)}, \quad f\in L^2(\R^d).
\end{align}
Here $\langle f,g\rangle_{L^2(\R^d)}=\int_{\R^d}\overline{f(x)}g(x)\,dx$ and, as usual\footnote{Such identification of kernels with operators will be made without further comment.},
\begin{align*}
  (P^{(\alpha)}(t,\cdot,\cdot)f)(x) := \int_{\R^d} P^{(\alpha)}(t,x,y)f(y)\,dy, \quad t>0,\,x\in\R^d.
\end{align*}
Moreover, using the Fourier transform
\begin{align*}
  \hat f(\xi):=(\F f)(\xi):=(2\pi)^{-d/2}\int_{\R^d}\me{-ix\cdot\xi}f(x)\,dx,
\end{align*}\index{$\hat f$}\index{$\F$}initially defined on $L^1(\R^d)$, and extended to a unitary operator on $L^2(\R^d)$, we get
\begin{align}
  \label{eq:dirichletformalpha}
  \begin{split}
    \ce^{(\alpha)}[f]
    & = \int_{\R^d} |\hat f(\xi)|^2|\xi|^\alpha\,d\xi.
  \end{split}
\end{align}
Recall the Sobolev spaces,
\begin{align}
  H^s(\R^d) := \left\{f\in L^2(\R^d):\,
  \int_{\R^d} |\xi|^{2s} |\hat f(\xi)|^2\,d\xi < \infty\right\}, \quad s\geq0,
\end{align}\index{$H^s(\R^d)$}
and the corresponding homogeneous Sobolev spaces $\dot H^s(\R^d)$\index{$\dot H^{s}(\R^d)$} for $s\geq0$, defined as the completion of $C_c^\infty(\R^d)$ under $\||\xi|^s \hat f\|_{L^2(\R^d)}$. Then we write, for $\alpha=2$,
\begin{align}
  \label{eq:deflaplacequadform}
  \ce^{(2)}[f]:=\langle\nabla f,\nabla f\rangle_{L^2(\R^d)}, \quad f\in \dot H^1(\R^d),
\end{align}\index{$\ce^{(2)}$}which is the quadratic form of $-\Delta$. Since each component of $D:=-i\nabla$ in $L^2(\R^d)$ is self-adjoint with domain $H^1(\R^d)$ and they commute, the spectral theorem ensures that $|D|^\alpha:=(-\Delta)^{\alpha/2}$ with domain $H^\alpha(\R^d)$ is also self-adjoint. Moreover, $|D|^\alpha$ coincides with the Friedrichs extension corresponding to the quadratic form $\ce^{(\alpha)}$, initially defined on $C_c^\infty(\R^d)$. In particular, the kernel in \eqref{eq:defusualheatkernel} is just the integral kernel of $\me{-t|D|^\alpha}$.

\subsection{Angular momentum decomposition}
\label{ss:angularmomenta}

Our goal is to analyze the heat kernel of $(-\Delta)^{\alpha/2}-\kappa/|x|^\alpha$, defined by functional calculus, associated with the quadratic form
\begin{align}
  \label{eq:defhardyop}
  \ce^{(\alpha)}[f] - \kappa \int_{\R^d}\frac{|f(x)|^2}{|x|^\alpha}\,dx,
\end{align}
by taking into account the spherical symmetry of the operators $|D|^\alpha$ and $|x|^{-\alpha}$. By spherical symmetry we mean that $|D|^\alpha$ and $|x|^{-\alpha}$ commute with the generator of rotations---sometimes called angular momentum operator---and therefore with its square, the Laplace--Beltrami operator on $L^2(\bs^{d-1})$. The latter is a non-negative operator with purely discrete spectrum and orthonormal eigenbasis consisting of spherical harmonics. The spherical symmetry allows to decompose $|D|^\alpha$ and $|x|^{-\alpha}$ into a direct sum of operators, each acting on radial functions multiplied by a spherical harmonic.
The following lemma provides an \emph{angular momentum decomposition} of $\ce^{(\alpha)}[f]$ and quadratic forms corresponding to spherically symmetric multiplication operators, such as $|x|^{-\alpha}$. It says that these forms break down into forms corresponding to different spaces $V_{\ell,m}$. To that end note that for every $f\in L^2(\R^d)$, there are unique coefficients $f_{\ell,m}\in V_{\ell,m}$, $\ell\in L_d$, $m\in M_\ell$, such that
\begin{align}
  \label{eq:l2expansionsphericalharmonics}
  f = \sum_{\ell\in L_d}\sum_{m\in M_\ell}f_{\ell,m}, \qquad
  \|f\|_{L^2(\R^d)}^2 = \sum_{\ell\in L_d}\sum_{m\in M_\ell}\|f_{\ell,m}\|_{L^2(\R^d)}^2,
\end{align}
see, e.g., \cite{BogdanMerz2024,Kalf1995}.

\begin{lemma}
  \label{fourierbesselcor}
  Let $d\!\in\!\N$, $\alpha\!\in\!(0,2]$, and $f\!\in\! H^{\frac\alpha2}(\R^d)$ with the expansion \eqref{eq:l2expansionsphericalharmonics}. Then,
  \begin{subequations}
    \begin{align}
      & \ce^{(\alpha)}[f] = \sum_{\ell\in L_d}\sum_{m\in M_\ell}\ce^{(\alpha)}[f_{\ell,m}]
        \quad \text{and} \quad \\
      & \int_{\R^d}\frac{|f(x)|^2}{|x|^\alpha}\,dx
        = \sum_{\ell\in L_d}\sum_{m\in M_\ell}\int_{\R^d}\frac{|f_{\ell,m}(x)|^2}{|x|^\alpha}\,dx.
    \end{align}
  \end{subequations}
\end{lemma}
The statement follows from the orthonormality of $\{Y_{\ell,m}\}_{\ell\in L_d,m\in M_\ell}$ in $L^2(\bs^{d-1})$, see, e.g., \cite[Lemma~1.2]{BogdanMerz2024}. For $\alpha\in(0,2)$, Lemma~\ref{fourierbesselcor} actually extends to all $f\in L^2(\R^d)$.

\smallskip
In view of the angular momentum decomposition and the parameterization of $\kappa$ in terms of $\Phi_{d_\ell}^{(\alpha)}(\eta)$, we will, from now on, instead of~\eqref{eq:defhardyop}, consider the form
\begin{align}
  \label{eq:hardyforml}
  \begin{split}
    \ce_{d_\ell,\eta}^{(\alpha)}[f]\index{$\ce_{d_\ell,\eta}^{(\alpha)}$}
    := \ce^{(\alpha)}[f] - \Phi_{d_\ell}^{(\alpha)}(\eta) \int_{\R^d}\frac{|f(x)|^2}{|x|^\alpha}\,dx.
  \end{split}
\end{align}
To be precise, for $\alpha\in(0,2\wedge d_\ell)$ and $\eta\in(-\alpha,(d_\ell-\alpha)/2]$, we consider $\ce_{d_\ell,\eta}^{(\alpha)}$ on the set $\{[f]_{\ell,m}:\, f\in \D(\ci_{(d_\ell-1)/2),\eta})\}$, where $\ci_{(d_\ell-1)/2,\eta}$ is the quadratic form \eqref{eq:defizetaeta} arising in the ground state representation of $\ce_{d_\ell,\eta}^{(\alpha)}$; see Theorem~\ref{relationhardyformheatkernel} and \eqref{eq:hardyremainderagaintransformedlimit}, which assert
\begin{align}
  \label{eq:hardyl}
  \ce_{d_\ell,\eta}^{(\alpha)}[[u]_{\ell,m}] = \ci_{(d_\ell-1)/2,\eta}[u]
\end{align}
for all $\ell\in L_d$, 
$\eta\in(-\alpha,(d_\ell-\alpha)/2]$, and 
$u\in L^2(\R_+,r^{d_\ell-1}dr)$. 
In Theorem~\ref{relationhardyformheatkernel} we also show that $\ci_{(d_\ell-1)/2,\eta}$ is nonnegative, closed, and symmetric on its maximal domain. The associated self-adjoint operator (see, e.g., \cite[Theorem~1.3.1]{Fukushimaetal2011}) is denoted by $\cl_{\Phi_{d_\ell}^{(\alpha)}(\eta),\ell}$---as introduced in \eqref{eq:angularmomentumdecomposition}--\eqref{eq:lkappaell} and after \eqref{eq:hardylinformal}---when $\zeta=(d_\ell-1)/2$. Correspondingly, for $\ell\in L_d$ and $\eta\in(-\alpha,(d_\ell-\alpha)/2]$, we denote the self-adjoint operator associated with $\ce_{d_\ell,\eta}^{(\alpha)}$ by
\begin{align}
  \label{eq:hardyopl}
  \left(|D|^\alpha-\frac{\Phi_{d_\ell}^{(\alpha)}(\eta)}{|x|^\alpha}\right)\big|_{V_\ell} \quad \text{in} \ V_\ell.
\end{align}
For $\alpha=2$, \eqref{eq:hardyopl} and \eqref{eq:hardyforml} denote $(U_\ell^*\cl_{\kappa,\ell} U_\ell)\otimes\one_{L^2(\bs^{d-1})}$ and the associated quadratic form, with the self-adjoint operator $\cl_{\kappa,\ell}$ defined after \eqref{eq:deflzeta} and the unitary operator $U_\ell:L^2(\R_+,r^{d-1}dr)\to L^2(\R_+,r^{d_\ell-1}dr)$ introduced after \eqref{eq:angularmomentumdecomposition}.

\section{Hardy perturbations of subordinated Bessel kernels}
\label{s:pointwise}

In this section, we recall well-known and more recent results on ordinary and subordinated Bessel heat kernels, as well as Schr\"odinger perturbations thereof, which are crucial to prove Theorem~\ref{mainresult} for $\alpha\in(0,2\wedge d_\ell)$. As we explain in the ensuing section, the naked subordinated Bessel heat kernels naturally arise when considering heat kernels of fractional Laplacians restricted to a fixed angular momentum channel. Thus, to prove Theorem~\ref{mainresult} for $\alpha<2$, we will show that Schr\"odinger perturbations of the bare subordinated Bessel heat kernels by Hardy potentials coincide with the heat kernel of the Hardy operator, restricted to a fixed angular momentum channel. We will achieve this goal in Section~\ref{s:quadraticforms}.
Our setting in the present section is slightly more general than needed for Theorem~\ref{mainresult} in that we allow the index $\zeta-1/2$ of the Bessel process to be any number in $(-1,\infty)$, while Theorem~\ref{mainresult} only requires $\zeta=(d_\ell-1)/2$.

\subsection{Subordinated Bessel heat kernels}
\label{s:semigroupproperties}

We recall properties of the Gaussian heat kernel
\begin{align}
  \me{t\Delta}(x,y)
  = P^{(2)}(t,x,y)
  = \frac{1}{(4\pi t)^{d/2}}\exp\left(-\frac{|x-y|^2}{4t}\right), \quad t>0,\ x,y\in\R^d,
\end{align}
and its $\frac\alpha2$-subordinated semigroups $\me{-t|D|^\alpha}$, restricted to $V_{\ell}$ for each $\ell\in L_d$. These restrictions lead, according to Proposition~\ref{relationformheatkernel} below, to kernels defined on $\R_+$, which depend on $d,\ell,\alpha$, but not on $m\in M_\ell$. We will, in fact, consider more general kernels \emph{indexed} by $\alpha\in(0,2]$ and an \emph{arbitrary} parameter $\zeta\in(-1/2,\infty)$. In the context of Theorem~\ref{mainresult}, we may think of $\zeta$ as equal to $(d_\ell-1)/2$. We then define and analyze these kernels for $\alpha=2$ and use subordination to study the case $\alpha<2$. To that end, recall that for $\alpha\in(0,2)$ and $t>0$, by Bernstein's theorem, the completely monotone function $[0,\infty)\ni\lambda\mapsto\me{-t\lambda^{\alpha/2}}$  is the Laplace transform of a probability density function $\R_+\ni\tau\mapsto\sigma_t^{(\alpha/2)}(\tau)$. That is,
\begin{align}
  \label{eq:subordination}
  \me{-t\lambda^{\alpha/2}} = \int_0^\infty \me{-\tau\lambda}\,\sigma_t^{(\alpha/2)}(\tau)\,d\tau, \quad t>0,\,\lambda\geq0,
\end{align}
see, e.g., \cite[Chapter~5]{Schillingetal2012} and \cite[Appendix~B]{BogdanMerz2024} for some useful properties of and sharp estimates of $\sigma_t^{(\alpha/2)}(\tau)$ and further references. We now introduce the Bessel and $\tfrac\alpha2$-subordinated Bessel kernels.

\begin{definition}
  \label{defsemigroupptwise}
  \label{defsemigroupptwisealpha}
  Let $\zeta\in(-1/2,\infty)$. We consider the reference (speed) measure $r^{2\zeta}dr$ on $\R_+$ and, for $t>0$, $r,s\in\R_+$, define the Bessel heat kernel
  \begin{align}
    \label{eq:defpheatalpha2}
    \begin{split}
      p_\zeta^{(2)}(t,r,s) & : = \frac{(rs)^{1/2-\zeta}}{2t}\exp\left(-\frac{r^2+s^2}{4t}\right)I_{\zeta-1/2}\left(\frac{rs}{2t}\right),
    \end{split}
  \end{align}
  where $I_{\zeta-1/2}$ is the modified Bessel function of the first kind. For $\alpha\in(0,2)$, we let
\begin{align}
    \label{eq:defpheatalpha}
    \begin{split}
      p_\zeta^{(\alpha)}(t,r,s) & : = \int_0^\infty p_\zeta^{(2)}(\tau,r,s)\,\sigma_t^{(\alpha/2)}(\tau)\,d\tau.
    \end{split}
  \end{align}\index{$p_\zeta^{(\alpha)}(t,r,s)$}  
\end{definition}

The kernel $p_\zeta^{(2)}(t,r,s)$ is the transition density of the Bessel process of order $\zeta-1/2$ \emph{reflected at the origin}. We remark that the Bessel process of order $\zeta-1/2$ \emph{killed at the origin} has the transition density \eqref{eq:defpheatalpha2}, but with $I_{\zeta-1/2}(\cdot)$ replaced with $I_{|\zeta-1/2|}(\cdot)$. When $\zeta\geq1/2$, the Bessel process never hits the origin, i.e., no condition (reflecting or killing) needs to be imposed at the origin. See, e.g., \cite{Maleckietal2016}. For more details about $p_\zeta^{(2)}$, see, e.g., the textbooks \cite[Part I, Section~IV.6 or Appendix~1.21]{BorodinSalminen2002} or \cite[Chapter~XI]{RevuzYor1999}.
Note that $p_\zeta^{(\alpha)}(t,r,s)$ is a probability transition density and a strongly continuous contraction semigroup on $L^2(\R_+,r^{2\zeta}dr)$.

\begin{proposition}
  \label{summarypropertiespzeta}
  If $\zeta\in(-1/2,\infty)$, $\alpha\in(0,2]$, and $t,t',r,s>0$, then $p_\zeta^{(\alpha)}(t,r,s)>0$,
  \begin{align}
    \label{eq:normalizedalpha}
    & \int_0^\infty p_\zeta^{(\alpha)}(t,r,s) s^{2\zeta}\,ds = 1, \\
    \label{eq:chapman}
    & \int_0^\infty p_\zeta^{(\alpha)}(t,r,z) p_\zeta^{(\alpha)}(t',z,s) z^{2\zeta}\,dz = p_\zeta^{(\alpha)}(t+t',r,s), \\
    \label{eq:scalingalpha}
    & p_\zeta^{(\alpha)}(t,r,s) = t^{-\frac{2\zeta+1}{\alpha}} p_\zeta^{(\alpha)}\left(1,\frac{r}{t^{1/\alpha}},\frac{s}{t^{1/\alpha}}\right),
  \end{align}
  and $\{p_\zeta(t,\cdot,\cdot)\}_{t>0}$ is a strongly continuous contraction semigroup on $L^2(\R_+,r^{2\zeta}dr)$.
\end{proposition}

For a proof of Proposition~\ref{summarypropertiespzeta}, see, e.g., \cite[Section~2]{BogdanMerz2024} and \cite[Lemma~1.4]{Bogdanetal2024}.

\smallskip
In the following, we recall the sharp upper and lower bounds for $p_\zeta^{(\alpha)}(t,r,s)$ proved in \cite{BogdanMerz2025}, and the explicit expression when $\alpha=1$; see, e.g., \cite[p.~136]{Betancoretal2010} for a computation.

\begin{proposition}
  \label{heatkernelalpha1subordinatedboundsfinal}
  Let $\zeta\in(-1/2,\infty)$. Then, for all $r,s,t>0$,
  \begin{subequations}
    \label{eq:easybounds2}
    \begin{align}
      p_\zeta^{(2)}(t,r,s)
      \label{eq:easybounds2a}
      & \asymp_\zeta t^{-\frac12}\frac{\exp\left(-\frac{(r-s)^2}{ct}\right)}{(rs+t)^{\zeta}} \\
      \label{eq:easybounds2b}
      & \asymp_\zeta \left(1\wedge\frac{r}{t^{1/2}}\right)^{\zeta}\, \left(1\wedge\frac{s}{t^{1/2}}\right)^{\zeta}\,\left(\frac{1}{rs}\right)^{\zeta} \cdot t^{-\frac12} \cdot \exp\left(-\frac{(r-s)^2}{c_\zeta t}\right).
    \end{align}
  \end{subequations}
  Moreover, for all $\alpha\in(0,2)$ and all $r,s,t>0$,
  \begin{align}
    \label{eq:heatkernelalpha1weightedsubordinatedboundsfinal}
    \begin{split}
      p_\zeta^{(\alpha)}(t,r,s)
      & \sim_{\zeta,\alpha} \frac{t}{|r-s|^{1+\alpha}(r+s)^{2\zeta} + t^{\frac{1+\alpha}{\alpha}}(t^{\frac1\alpha}+r+s)^{2\zeta}}.
    \end{split}
  \end{align}
  For $\alpha=1$, we have, for all $r,s,t>0$,
  \begin{align}
    \label{eq:heatkernell}
    \begin{split}
      p_\zeta^{(1)}(t,r,s)
      & = \frac{2\Gamma(\zeta+1)}{\sqrt\pi} \cdot \frac{t}{\left(r^2+s^2+t^2\right)^{\zeta+1}} \\
        & \quad \times\, _2\tilde{F}_1\left(\zeta+1,\zeta+2;\zeta+\frac{1}{2};\frac{4r^2s^2}{\left(r^2+s^2+t^2\right)^2}\right).
    \end{split}
  \end{align}
  Moreover, considering $\zeta\in\{0,1\}$, we et
  \begin{subequations}
    \label{eq:heatkernell0}
    \begin{align}
      p_0^{(1)}(t,r,s) & = \frac{2 t}{\pi  \left(r^2+s^2+t^2\right) \left(1- 4 r^2 s^2 \cdot \left(r^2+s^2+t^2\right)^{-2}\right)}, \\
      p_1^{(1)}(t,r,s) & = \frac{4}{\pi}\, \frac{t}{(r^2-s^2)^2+t^2(t^2+2r^2+2s^2)}.
    \end{align}
  \end{subequations}
\end{proposition}

Using the explicit expressions for the $\alpha/2$-stable subordination density for all rational $\alpha/2$ \cite{PensonGorska2010}, one could in principle compute $p_\zeta^{(\alpha)}(t,r,s)$ also in these cases.

\subsection{Hardy-type Schr\"odinger perturbation}
\label{ss:defperturbedsemigroup}

We now present the Schr\"odinger perturbation of $p_\zeta^{(\alpha)}$ by the Hardy potential, first constructed in \cite{Bogdanetal2024}.
The following results are only used for $\alpha<2$. Thus, even though some of them also hold for $\alpha=2$, we choose to consider only $\alpha<2$ for a coherent presentation.
For $\zeta\in(-1/2,\infty)$, $\alpha\in(0,2)\cap(0,2\zeta+1)$, and $\eta\in(-\alpha,2\zeta+1)$, let
\begin{align}
  \label{eq:defpsietazetaOLD}
  \Psi_\zeta(\eta) := \frac{2^\alpha\Gamma\left(\frac{2\zeta+1-\eta}{2}\right)\Gamma\left(\frac{\alpha+\eta}{2}\right)}{\Gamma\left(\frac{\eta}{2}\right)\Gamma\left(\frac{2\zeta+1-\eta-\alpha}{2}\right)}
  \quad \text{and} \quad
  q(z) := \frac{\Psi_\zeta(\eta)}{z^\alpha} \quad \text{for} \ z>0.
\end{align}\index{$\Psi_\zeta(\eta)$}\index{$q(z)$}
In particular, $\Psi_\zeta(\eta)$ is symmetric around $\eta=(2\zeta+1-\alpha)/2$, strictly positive for $\eta\in[0,2\zeta+1-\alpha]$, zero if $\eta\in\{0,2\zeta+1-\alpha\}$, and tends to $-\infty$ as $\eta\to-\alpha$ or $\eta\to2\zeta+1$. To make a connection to the Introduction, note that $\Psi_{(d_\ell-1)/2}(\eta)=\Phi_{d_\ell}^{(\alpha)}(\eta)$.

\begin{definition}
  \label{deffeynmankactransformed}
  For $\zeta\in(-1/2,\infty)$, $\alpha\in(0,2)\cap(0,2\zeta+1)$, and $\eta\in (-\alpha,\frac{2\zeta+1-\alpha}{2}]$, let $p_{\zeta,\eta}^{(\alpha)}$ be the Schr\"odinger perturbation of $p_\zeta^{(\alpha)}$ by $q$, defined in \cite[Definition~1.5]{Bogdanetal2024}.
\end{definition}

\begin{remark}
For \( \eta > 0 \), \( p_{\zeta, \eta}^{(\alpha)} \) can be given explicitly via the perturbation series,
  \begin{align}
    \label{eq:feynmankactransformed}
    \begin{split}
      p_{\zeta,\eta}^{(\alpha)}(t,r,s)
      & := \sum_{n\geq0} p_t^{(n,D)}(r,s), \quad \mbox{where}\quad 
        p_t^{(0,D)}(r,s) := p_\zeta^{(\alpha)}(t,r,s), \\
      p_t^{(n,D)}(r,s) & := \int_0^t d\tau \int_0^\infty dz\, z^{2\zeta} p_\zeta^{(\alpha)}(t,r,z) q(z) p_{t-\tau}^{(n-1,D)}(z,s), \quad n\in\N.
    \end{split}
  \end{align}\index{$p_{\zeta,\eta}^{(\alpha)}$}Here, each term $p_t^{(n,D)}(r,s)$ may be understood as an iteration of \emph{Duhamel's} formula below, hence the superscript $D$.
\end{remark}

The following results are proved in \cite{Bogdanetal2024}.

\begin{theorem}
  \label{propertiesschrodheatkernel}
  Let $\zeta\in(-1/2,\infty)$, $\alpha\in(0,2)\cap(0,2\zeta+1)$, and $\eta\in(-\alpha,\frac{2\zeta+1-\alpha}{2}]$, $r,s,t,t'>0$. Then the following statements hold.\\
  \textup{(1)}
  We have the Chapman--Kolmogorov equation
  \begin{align}
    \int_0^\infty dz\, z^{2\zeta} p_{\zeta,\eta}^{(\alpha)}(t,r,z) \, p_{\zeta,\eta}^{(\alpha)}(t',z,s)
    & = p_{\zeta,\eta}^{(\alpha)}(t+t',r,s).
  \end{align}
  \textup{(2)}
  We have the Duhamel formulae
  \begin{align}
    \label{eq:duhamelclassictransformed}
    \begin{split}
      p_{\zeta,\eta}^{(\alpha)}(t,r,s)
      & = p_\zeta^{(\alpha)}(t,r,s) + \int_0^t d\tau \int_0^\infty dz\, z^{2\zeta} p_\zeta^{(\alpha)}(\tau,r,z) q(z)\, p_{\zeta,\eta}^{(\alpha)}(t-\tau,z,s) \\
      & = p_\zeta^{(\alpha)}(\tau,r,s) + \int_0^t d\tau \int_0^\infty dz\, z^{2\zeta} p_{\zeta,\eta}^{(\alpha)}(\tau,r,z) q(z)\, p_\zeta^{(\alpha)}(t-\tau,z,s).
    \end{split}
  \end{align}
  \textup{(3)}
  We have the scaling relation
  \begin{align}
    \label{eq:scalingalphahardy}
    p_{\zeta,\eta}^{(\alpha)}(t,r,s)
    = t^{-\frac{2\zeta+1}{\alpha}} p_{\zeta,\eta}^{(\alpha)}\left(1,\frac{r}{t^{1/\alpha}},\frac{s}{t^{1/\alpha}}\right).
  \end{align}
  \textup{(4)}
  The function $h(r)=r^{-\eta}$ is invariant under $p_{\zeta,\eta}^{(\alpha)}$, i.e.,
  \begin{align}
    \label{eq:supermediantransformedalphaintro}
    \begin{split}
      \int_0^\infty ds\, s^{2\zeta} p_{\zeta,\eta}^{(\alpha)}(t,r,s) h(s) = h(r), \quad t>0.
    \end{split}
  \end{align}
  \textup{(5)}
  For $\eta>0$, $p_\zeta^{(\alpha)}(t,r,s)\leq p_{\zeta,\eta}^{(\alpha)}(t,r,s)$, and for $\eta<0$, $p_\zeta^{(\alpha)}(t,r,s)\geq p_{\zeta,\eta}^{(\alpha)}(t,r,s)$.
  \\
  \textup{(6)}
  $\{p_{\zeta,\eta}^{(\alpha)}(t,\cdot,\cdot)\}_{t\geq0}$ is a strongly continuous semigroup of contractions on $L^2(\R_+,r^{2\zeta}dr)$.
\end{theorem}

The proof of Theorem~\ref{mainresult} for $\alpha\in(0,2\wedge d_\ell)$ essentially relies on the following sharp pointwise bounds for $p_{\zeta,\eta}^{(\alpha)}(t,r,s)$.

\begin{theorem}[{\cite[Theorem~1.1]{Bogdanetal2024}}]
  \label{mainresultgen}
  Let $\zeta\in(-1/2,\infty)$, $\alpha\in(0,2)\cap(0,2\zeta+1)$, and $\eta\in(-\alpha,\frac{2\zeta+1-\alpha}{2}]$. Then, $p_{\zeta,\eta}^{(\alpha)}(t,r,s)$ is jointly continuous as a function of $r,s,t>0$,
  \begin{align}
    \label{eq:mainresultgenalpha}
      p_{\zeta,\eta}^{(\alpha)}(t,r,s)
      \sim_{\zeta,\alpha,\eta} \left(1\wedge\frac{r}{t^{1/\alpha}}\right)^{-\eta}\, \left(1\wedge\frac{s}{t^{1/\alpha}}\right)^{-\eta}\, p_\zeta^{(\alpha)}(t,r,s) \quad \text{for} \ \alpha<2,
  \end{align}
  and
  \begin{align}
    \label{eq:mainresultgen2}
    p_{\zeta,\eta}^{(2)}(t,r,s)
    \asymp_{\zeta,\eta} \left(1\wedge\frac{r}{t^{1/2}}\right)^{-\eta}\, \left(1\wedge\frac{s}{t^{1/2}}\right)^{-\eta}\, p_\zeta^{(2)}(t,r,s).
  \end{align}
\end{theorem}

We, however, need to make a connection between the kernel $p_{\zeta,\eta}^{(\alpha)}$ and suitable quadratic forms. This is the main subject of this paper and the focus of the next section.

\section{Quadratic forms}
\label{s:quadraticforms}

In this section, we demonstrate that the quadratic form generated by \( p_{(d_\ell - 1)/2, \eta}^{(\alpha)} \) coincides with the radial part of the form \( \ce_{d_\ell, \eta}^{(\alpha)} \). We state this result precisely in the main result of this section, Theorem~\ref{relationhardyformheatkernel} below. Along with Theorem~\ref{mainresultgen}, it serves as a crucial ingredient in proving Theorem~\ref{mainresult} for $\alpha\in(0,2\wedge d_\ell)$.
The following results are only used for $\alpha<2$. As in the previous section, even though some of them also hold for $\alpha=2$, we will consider only $\alpha<2$ for a coherent presentation.

\subsection{Definition of quadratic forms on $\R_+$}
\label{s:defquadformsemigroup}

We first define the relevant quadratic forms on $\R_+$.
For $\zeta\in(-1/2,\infty)$, $\alpha\in(0,2)$, and $r,s\in\R_+$ with $r\neq s$, let
\begin{align}
  \label{eq:defnuell1}
  \begin{split}
    \nu_{\zeta}(r,s)
    & := \lim_{t\searrow0}\frac{p_\zeta^{(\alpha)}(t,r,s)}{t}
      = \lim_{t\searrow0} \int_0^\infty p_\zeta^{(2)}(\tau,r,s)\frac{\sigma_t^{(\alpha/2)}(\tau)}{t}\,d\tau \\
    & \ = 2^{1+\alpha} \frac{\Gamma\left(\frac\alpha2+1\right)\,\sin\left(\frac{\pi\alpha}{2}\right)}{\pi}\, \Gamma\left(\zeta+\frac\alpha2+\frac12\right)\, (r^2+s^2)^{-(\zeta+\frac\alpha2+\frac12)} \\
    & \qquad \times\, _2\tilde F_1\left(\frac{\zeta+\alpha/2+1/2}{2}, \frac{\zeta+\alpha/2+3/2}{2}; \zeta+\frac12; \frac{4(rs)^2}{(r^2+s^2)^2}\right) \\
    & \ \sim_{\zeta,\alpha} \frac{1}{(r+s)^{2\zeta} \cdot |r-s|^{1+\alpha}}
  \end{split}
\end{align}\index{$\nu_\zeta(r,s)$}
be the L\'evy kernel associated to $p_\zeta^{(\alpha)}(t,r,s)$. The kernel obeys
\begin{align}
  \label{eq:defnuell4}
  \frac{p_\zeta^{(\alpha)}(t,r,s)}{t}
  \lesssim_{\zeta,\alpha} \nu_{\zeta}(r,s), \quad r,s,t>0.
\end{align}
See \cite[Proposition~2.4]{BogdanMerz2024} for proofs of \eqref{eq:defnuell1} and \eqref{eq:defnuell4}.
By the general theory \cite[Section~1.3]{Fukushimaetal2011}, the quadratic form associated to $p_\zeta^{(\alpha)}(t,\cdot,\cdot)$,
\begin{align}
  \label{eq:defezeta}
  \begin{split}
    \ce_\zeta[u]
    & := \lim_{t\to0}\frac1t\langle u,u-p_\zeta^{(\alpha)}(t,\cdot,\cdot)u\rangle_{L^2(\R_+,r^{2\zeta}dr)},\quad u\in L^2(\R_+,r^{2\zeta}dr),
  \end{split}
\end{align}\index{$\ce_\zeta$}is well-defined for all $\alpha\in(0,2)$. We denote the maximal domain of $\ce_\zeta$ by
\begin{align}
  \label{eq:maximaldomainezeta}
  \D(\ce_\zeta) := \{v\in L^2(\R_+,r^{2\zeta}dr):\,\ce_\zeta[v]<\infty\}.
\end{align}
Moreover,
\begin{align}
  \label{eq:ezetalimit}
  \ce_\zeta[u] = \int_0^\infty dr\, r^{2\zeta} \int_0^\infty ds\, s^{2\zeta} |u(r)-u(s)|^2\,\nu_{\zeta}(r,s),
\end{align}
which follows from \eqref{eq:defnuell4} and the Dominated Convergence Theorem if the right-hand side is finite, or---in the opposite case---by Fatou's lemma.
By the Hardy inequality in \cite{BogdanMerz2024}, the form sum
\begin{align}
  \ce_\zeta[u] - \Psi_\zeta(\eta)\int_{\R_+}|u(r)|^2\, r^{2\zeta-\alpha}\,dr,
  \quad u\in C_c^\infty(\R_+),
\end{align}
is non-negative for all $\eta\in(-\alpha,\frac{2\zeta+1-\alpha}{2}]$. Moreover, for $h(r)=r^{-\eta}$, we have the ground state representation
\begin{align}
  \label{eq:hardyremainderagaintransformedlimit}
  \begin{split}
    \ce_\zeta[u]
    & = \ci_{\zeta,\eta}[u] + \Psi_\zeta(\eta)\int_{\R_+} \frac{|u(r)|^2}{r^\alpha}\,r^{2\zeta}\,dr
  \end{split}
\end{align}
with 
\begin{align}
  \label{eq:defizetaeta}
  \begin{split}
    \ci_{\zeta,\eta}[u]
    & := \frac12\int_0^\infty dr\, r^{2\zeta}\int_0^\infty ds\, s^{2\zeta}\, \nu_{\zeta}(r,s) \left|\frac{u(r)}{h(r)} - \frac{u(s)}{h(s)}\right|^2 h(r) h(s),
  \end{split}
\end{align}\index{$\ci_{\zeta,\eta}$}
which is finite on the maximal domain
\begin{align}
  \D(\ci_{\zeta,\eta}) := \{v\in L^2(\R_+,r^{2\zeta}dr):\,\ci_{\zeta,\eta}[v]<\infty\}.
\end{align}
Formula \eqref{eq:hardyremainderagaintransformedlimit} holds for all $\alpha\in(0,2\wedge(2\zeta+1))$ and $u\in L^2(\R_+,r^{2\zeta}dr)$; see \cite[Theorem~3.2]{BogdanMerz2024} for $\eta\in[0,(2\zeta+1-\alpha)]$ and Lemma~\ref{gsreprrepulsive} below for $\eta\in(-\alpha,0)$.

In Theorem~\ref{relationhardyformheatkernel} below, we show that $\ci_{\zeta,\eta}$ is nonnegative, closed, symmetric for all $\alpha\in(0,2\wedge(2\zeta+1))$ and $\eta\in(-\alpha,(2\zeta+1-\alpha)/2]$. The corresponding self-adjoint operator defined by a theorem of Friedrichs \cite[Theorem~1.3.1]{Fukushimaetal2011} is denoted by $\cl_{\Phi_{d_\ell}^{(\alpha)}(\eta),\ell}$---as introduced in \eqref{eq:angularmomentumdecomposition}--\eqref{eq:lkappaell} and after \eqref{eq:hardylinformal}---when $\zeta=(d_\ell-1)/2$. It is for this operator that we prove the heat kernel bounds in Theorem~\ref{mainresult}. To this end, we will relate the form $\ci_{\zeta,\eta}$ to the quadratic form canonically associated to $p_{\zeta,\eta}^{(\alpha)}(t,\cdot,\cdot)$, 
\begin{align}
  \label{eq:defezetaeta}
  \begin{split}
    \ce_{\zeta,\eta}[u]
    & := \lim_{t\to0}\frac1t\langle u,u-p_{\zeta,\eta}^{(\alpha)}(t,\cdot,\cdot)u\rangle_{L^2(\R_+,r^{2\zeta}dr)},\quad u\in L^2(\R_+,r^{2\zeta}dr)
  \end{split}
\end{align}\index{$\ce_{\zeta,\eta}$}on its maximal domain
\begin{align}
  \D(\ce_{\zeta,\eta}) := \{v\in L^2(\R_+,r^{2\zeta}dr):\,\ce_{\zeta,\eta}[v]<\infty\}.
\end{align}
By \eqref{eq:defezeta}, $\ce_{\zeta,0}=\ci_{\zeta,0}=\ce_\zeta$, but the case of $\eta\neq 0$ needs to be clarified.
In the following, we prove that $\ce_{\zeta,\eta}$ is a nonnegative, closed symmetric form. Hence, by general theory \cite[Section~1.3]{Fukushimaetal2011}, there is a corresponding nonnegative, self-adjoint operator, called the generator of $p_{\zeta,\eta}^{(\alpha)}$. Secondly, we show that $\ci_{\zeta,\eta}=\ce_{\zeta,\eta}$ on $L^2(\R_+,r^{2\zeta}\,dr)$ which implies that the self-adjoint operators associated with $\ci_{\zeta,\eta}$ and $\ce_{\zeta,\eta}$ coincide. In particular, their heat kernels are $p_{\zeta,\eta}^{(\alpha)}$. This is the roadmap to prove our main result, Theorem~\ref{mainresult}, using the bounds for $p_{\zeta,\eta}^{(\alpha)}$ in Theorem~\ref{mainresultgen}.

\subsection{Connection between the forms $\ce_{d_\ell,\eta}^{(\alpha)}$, $\ci_{\zeta,\eta}$, and $\ce_{\zeta,\eta}$}
\label{s:notation}

In this subsection, we state the equality between the forms $\ci_{\zeta,\eta}$ and $\ce_{\zeta,\eta}$ 
and their connection to the form $\ce_{d_\ell,\eta}^{(\alpha)}$; see Theorem~\ref{relationhardyformheatkernel}. To this end, we first recall, after \cite[Proposition~4.3]{BogdanMerz2024}, the following connection of the forms $\ce^{(\alpha)}$ on $V_{\ell,m}$ and $\ce_{\zeta}$ on $L^2(\R_+,r^{2\zeta}dr)$.

\begin{proposition}
  \label{relationformheatkernel}
  Let $d\in\N$, $\alpha\in(0,2)$, $\ell\in L_d$, $m\in M_\ell$, $\zeta=(d_\ell-1)/2$.
  Then, for all $u\in L^2(\R_+,r^{d_\ell-1}dr)$, we have
  \begin{align}
    \label{eq:linksemigroupsrdrplus}
    \begin{split}
      \langle[u]_{\ell,m},P^{(\alpha)}(t,\cdot,\cdot)[u]_{\ell,m}\rangle_{L^2(\R^d)}
      & = \langle u, p_\zeta^{(\alpha)}(t,\cdot,\cdot)u\rangle_{L^2(\R_+,r^{2\zeta}dr)},
        \quad t>0
    \end{split}
  \end{align}
  and
  \begin{align}
    \label{eq:relationformheatkernel3}
    \ce^{(\alpha)}[[u]_{\ell,m}] = \ce_{\zeta}[u].
  \end{align}
\end{proposition}

The following result is essential to prove Theorem~\ref{mainresult} for $\alpha<2$. It identifies the forms $\ci_{\zeta,\eta}$ and $\ce_{\zeta,\eta}$ of \eqref{eq:defizetaeta} and \eqref{eq:defezetaeta} and their corresponding self-adjoint operators and heat kernels.

\begin{theorem}
  \label{relationhardyformheatkernel}
  Let $\zeta\in(-1/2,\infty)$,
  $\alpha\in(0,2\wedge(2\zeta+1))$,
  $\eta\in(-\alpha,\frac{2\zeta+1-\alpha}{2})$, and $h(r)=r^{-\eta}$.
  Then the following statements hold.
  \\
  \textup{(1)}
  If $d\in\N$, $\ell\in L_d$, $\alpha\in(0,2\wedge d_\ell)$, $\eta\in(-\alpha,(d_\ell-\alpha)/2]$, and $\zeta=(d_\ell-1)/2$, then
  \begin{align}
    \label{eq:hardyopformradial}
    \ce_{d_\ell,\eta}^{(\alpha)}[[u]_{\ell,m}] = \ci_{\zeta,\eta}[u], \quad u \in \D(\ci_{\zeta,\eta})
  \end{align}
  holds 
  for all $u\in L^2(\R_+,r^{d_\ell-1}dr)$.
  \\
  \textup{(2)}
  $\D(\ce_{\zeta,\eta})$ is a Hilbert space and $C_c^\infty(\R_+)$ is dense in $\D(\ce_{\zeta,\eta})$ in the norm $\sqrt{\ce_{\zeta,\eta}[\cdot]}+\|\cdot\|_{L^2(\R_+,r^{2\zeta}dr)}$.
  Moreover, $\ci_{\zeta,\eta} = \ce_{\zeta,\eta}$ on $L^2(\R_+, r^{2\zeta}dr)$, their domains and associated self-adjoint operators coincide, and their heat kernel is $p_{\zeta,\eta}^{(\alpha)}$.
\end{theorem}


\begin{proof}
  (1):
  Formula \eqref{eq:hardyopformradial} follows from Formulae \eqref{eq:hardyremainderagaintransformedlimit} and \eqref{eq:relationformheatkernel3}, and the equality $\langle[u]_{\ell,m},|x|^{-\alpha}[u]_{\ell,m}\rangle_{L^2(\R^d)}=\langle u,r^{-\alpha}u\rangle_{L^2(\R_+,r^{2\zeta}dr)}$.
  \\
  (2):
  These claims follow from Theorems~\ref{bogthm51} and~\ref{jakprop318} below when $\eta\in[0,(2\zeta+1-\alpha)/2]$ and $\eta\in(-\alpha,0]$, respectively.
\end{proof}

As outlined at the end of Section~\ref{ss:angularmomenta}, we can associate a self-adjoint operator to the nonnegative, symmetric, closed form $\ce_{d_\ell,\eta}^{(\alpha)}$, which we denote by $((-\Delta)^{\alpha/2}-\Phi_{d_\ell}^{(\alpha)}(\eta)/|x|^\alpha)|_{V_\ell}$ in $V_\ell$. The following formula for its heat kernel---defined by functional calculus---integrated against spherical harmonics provides a different perspective on Theorem~\ref{mainresult}.

\begin{corollary}
  \label{heatkernelintegratedsphericalharmonic}
  Let $d\in\N$, $\ell\in L_d$, $\alpha\in(0,2\wedge d_\ell)$, and $\eta\in(-\alpha,(d_\ell-\alpha)/2]$. Then, for all $r,s,t>0$,
  \begin{align}
    \label{eq:linksemigroupsrdrpluspointwisekappa}
    \begin{split}
    & (rs)^{-\ell} \!\!\!\!\!\iint\limits_{\bs^{d-1}\times\bs^{d-1}} \overline{Y_{\ell,m}(\omega_x)} Y_{\ell,m}(\omega_y) \exp\left[-t\left(|D|^\alpha-\frac{\Phi_{d_\ell}^{(\alpha)}(\eta)}{|x|^\alpha}\right)\Big|_{V_\ell}\right](r\omega_x,s\omega_y)\,d\omega_x\,d\omega_y \\
    & \quad = p_{(d_\ell-1)/2,\eta}^{(\alpha)}(t,r,s).
    \end{split}
  \end{align}
\end{corollary}

\begin{proof}
  This follows from Theorem~\ref{relationhardyformheatkernel}, once we show that the left-hand side of \eqref{eq:linksemigroupsrdrpluspointwisekappa} equals the heat kernel of the self-adjoint operator associated with $\ci_{(d_\ell-1)/2,\eta}$. We denote the self-adjoint operators associated with $\ci_{(d_\ell-1)/2,\eta}$ and $\ce_{d_\ell,\eta}^{(\alpha)}$ by $A_r\!=\!\cl_{\Phi_{d_\ell}^{(\alpha)}(\eta),\ell}$ in $L^2(\R_+,r^{2\zeta}dr)$ and $A=\left(|D|^\alpha-\Phi_{d_\ell}^{(\alpha)}(\eta)/|x|^\alpha\right)|_{V_\ell}$ in $V_\ell$, respectively\footnote{We choose this temporary notation only for a simpler presentation.}. For $\ell\in L_d$ and $m\in M_\ell$, we denote by
  $\Pi_{\ell,m}:=\one_{L^2(\R_+,r^{d-1}dr)}\otimes |Y_{\ell,m}\rangle\langle Y_{\ell,m}|$ the orthogonal projection onto $V_{\ell,m}$. Then,
  \begin{align}
    \Pi_{\ell,m} A \Pi_{\ell,m}
    = U_\ell^* A_rU_\ell \otimes |Y_{\ell,m}\rangle\langle Y_{\ell,m}|, \quad m\in M_\ell,
  \end{align}
  with the unitary operator $U_\ell:L^2(\R_+,r^{d-1}dr)\to L^2(\R_+,r^{d_\ell-1}dr)$ defined by $(U_\ell u)(r)=r^{-\ell}u(r)$. By the spherical symmetry of $A$, we have
  \begin{align}
    \Pi_{\ell,m} \me{-tA} \Pi_{\ell,m}
    = \exp\left(-t \Pi_{\ell,m} A \Pi_{\ell,m}\right)
    = U_\ell^* \me{-tA_r}U_\ell \otimes |Y_{\ell,m}\rangle\langle Y_{\ell,m}|,
    \quad m\in M_\ell.
  \end{align}
  Integrating this identity against $\overline{Y_{\ell,m}(\omega_x)} Y_{\ell,m}(\omega_y)\,d\omega_x\,d\omega_y$ yields
  \begin{align}
    \me{-tA_r}(r,s)
    = (rs)^{-\ell} \iint\limits_{\bs^{d-1}\times\bs^{d-1}} \overline{Y_{\ell,m}(\omega_x)} Y_{\ell,m}(\omega_y) \me{-tA}(r\omega_x,s\omega_y)\,d\omega_x\,d\omega_y,
  \end{align}
  as claimed.
\end{proof}

\subsection{Equality of $\ci_{\zeta,\eta}$ and $\ce_{\zeta,\eta}$ for $\eta\in(0,(2\zeta+1-\alpha)/2]$}
\label{s:connectiongenerators}

In this subsection, we 
fix $\zeta\in(-1/2,\infty)$ and $\alpha\in(0,2\wedge(2\zeta+1))$, and prove Theorem~\ref{relationhardyformheatkernel}(2) for all $\eta\in(0,(2\zeta+1-\alpha)/2]$.
Thus, our goal is to show  that $\D(\ce_{\zeta,\eta})=\D(\ci_{\zeta,\eta})$ and  $\ce_{\zeta,\eta}=\ci_{\zeta,\eta}$.

Here is the main result of this subsection.

\begin{theorem}
  \label{bogthm51}
  Let $\zeta\in(-1/2,\infty)$, $\alpha\in(0,2\wedge(2\zeta+1))$, and $\eta\in[0,(2\zeta+1-\alpha)/2]$. Then, $(\D(\ce_{\zeta,\eta}),\sqrt{\ce_{\zeta,\eta}[\cdot]}+\|\cdot\|_{L^2(\R_+,r^{2\zeta}dr)})$ is a Hilbert space.
  Moreover, $C_c^\infty(\R_+)$ is dense in $\D(\ce_{\zeta,\eta})$ with the norm $\sqrt{\ce_{\zeta,\eta}[\cdot]}+\|\cdot\|_{L^2(\R_+,r^{2\zeta}dr)}$.
  Furthermore, $\ce_{\zeta,\eta}=\ci_{\zeta,\eta}$ on $L^2(\R_+,r^{2\zeta}dr)$.
\end{theorem}

\subsubsection{Equality of forms on smaller domains}

We first show the equality of $\ce_{\zeta,\eta}$ and $\ci_{\zeta,\eta}$ on $\D(\ce_\zeta)$, see \eqref{eq:maximaldomainezeta}.

\begin{lemma}[{\cite[Lemma~5.1]{Bogdanetal2019}}]
  \label{boglem51}
  Let $\zeta\in(-1/2,\infty)$, $\alpha\in(0,2\wedge(2\zeta+1))$, and $\eta\in(0,(2\zeta+1-\alpha)/2]$. Then, $\D(\ce_\zeta)\subseteq\D(\ce_{\zeta,\eta})\subseteq\D(\ci_{\zeta,\eta})$ and
  \begin{align}
    \label{eq:boglem51}
    \ce_{\zeta,\eta}[v] = \ci_{\zeta,\eta}[v] = \ce_\zeta[v] - \Psi_\zeta(\eta)\int_0^\infty dr\, r^{2\zeta-\alpha} |v(r)|^2, \quad v\in\D(\ce_\zeta).
  \end{align}
\end{lemma}

\begin{proof}
  We first show that 
    \begin{align}\label{e.cEI}
    \ce_{\zeta,\eta}[v] \geq \ci_{\zeta,\eta}[v], \quad v\in L^2(\R_+,r^{2\zeta}dr).
  \end{align}
  To this end, recall from Proposition~\ref{summarypropertiespzeta} that $p_\zeta^{(\alpha)}$ is a strongly continuous, symmetric semigroup on $L^2(\R_+,r^{2\zeta}dr)$. By \eqref{eq:ezetalimit},
  \begin{align}
    \ce_{\zeta}(t)[v]
    := \frac1t\langle v,(1-p_\zeta^{(\alpha)}(t,\cdot,\cdot))v\rangle_{L^2(\R_+,r^{2\zeta}dr)}
    \stackrel{t\to0}{\longrightarrow} \ce_\zeta[v], \quad v\in L^2(\R_+,r^{2\zeta}dr).
  \end{align}
  By Theorem~\ref{propertiesschrodheatkernel}(6), $p_{\zeta,\eta}^{(\alpha)}(t,\cdot,\cdot)$ is also a contraction on $L^2(\R_+,r^{2\zeta}dr)$. Next, let $v\in L^2(\R_+,r^{2\zeta}dr)$ and, for $t>0$,
  \begin{align}
    \begin{split}
      \ce_{\zeta,\eta}(t)[v]
      & := \frac1t\langle v,(1-p_{\zeta,\eta}^{(\alpha)}(t,\cdot,\cdot))v\rangle_{L^2(\R_+,r^{2\zeta}dr)} \\
      & \ = \frac1t\left[\langle v,v\rangle - \langle p_{\zeta,\eta}^{(\alpha)}(t/2,\cdot,\cdot)v, p_{\zeta,\eta}^{(\alpha)}(t/2,\cdot,\cdot)v\rangle_{L^2(\R_+,r^{2\zeta}dr)}\right].
    \end{split}
  \end{align}
  Theorem~\ref{propertiesschrodheatkernel}(4) asserts that $p_{\zeta,\eta}^{(\alpha)}(t,\cdot,\cdot)h=h$ for $h(r)=r^{-\eta}$ and all $t>0$. Therefore,
  \begin{align}
    v(r) - (p_{\zeta,\eta}^{(\alpha)}(t,\cdot,\cdot)v)(r)
    = \int_0^\infty ds\, s^{2\zeta} p_{\zeta,\eta}^{(\alpha)}(t,r,s)\left(\frac{v(r)}{h(r)}-\frac{v(s)}{h(s)}\right)h(s)
  \end{align}
  and so
  \begin{align}
    \begin{split}
      & \langle v,v- p_{\zeta,\eta}^{(\alpha)}(t,\cdot,\cdot)v\rangle_{L^2(\R_+,r^{2\zeta}dr)} \\
      & \quad = \int_0^\infty dr\int_0^\infty ds\, (rs)^{2\zeta} p_{\zeta,\eta}^{(\alpha)}(t,r,s)\left(\frac{v(r)}{h(r)}-\frac{v(s)}{h(s)}\right) \cdot \frac{v(r)}{h(r)}h(r)h(s) \\
      & \quad = \int_0^\infty dr\int_0^\infty ds\, (rs)^{2\zeta} p_{\zeta,\eta}^{(\alpha)}(t,r,s)\left(\frac{v(s)}{h(s)}-\frac{v(r)}{h(r)}\right)\cdot\frac{v(s)}{h(s)}h(r)h(s),
    \end{split}
  \end{align}
  where we used $p_{\zeta,\eta}^{(\alpha)}(t,r,s)=p_{\zeta,\eta}^{(\alpha)}(t,s,r)$ to deduce the last equality. Therefore,
  \begin{align}
    \begin{split}
      \ce_{\zeta,\eta}(t)[v]
      & = \frac1t\langle v,(v- p_{\zeta,\eta}^{(\alpha)}(t,\cdot,\cdot)v)\rangle_{L^2(\R_+,r^{2\zeta}dr)} \\
      & = \frac{1}{2t}\iint_{\R_+\times\R_+} dr\,ds\, (rs)^{2\zeta} p_{\zeta,\eta}^{(\alpha)}(t,r,s)\left(\frac{v(r)}{h(r)}-\frac{v(s)}{h(s)}\right)^2 h(r)h(s) \\
      & \geq \frac12\iint_{\R_+\times\R_+} dr\,ds\, (rs)^{2\zeta} \frac{p_\zeta^{(\alpha)}(t,r,s)}{t} \left(\frac{v(r)}{h(r)}-\frac{v(s)}{h(s)}\right)^2 h(r)h(s),
    \end{split}
  \end{align}
  where we used $p_{\zeta,\eta}^{(\alpha)}(t,r,s)\geq p_\zeta^{(\alpha)}(t,r,s)$.
  Thus, by taking the limit $t\to0$, Fatou's lemma, the definition \eqref{eq:defnuell1}, and the bound \eqref{eq:defnuell4}, we infer the first desired inequality \eqref{e.cEI},
  which in particular implies $\D(\ce_{\zeta,\eta})\subseteq\D(\ci_{\zeta,\eta})$.

  We next show the converse inequality $\ce_{\zeta,\eta}[v] \leq \ci_{\zeta,\eta}[v]$. By \eqref{e.cEI}, the definition of $q$ in \eqref{eq:defpsietazetaOLD}, and $\ci_{\zeta,\eta}$ in \eqref{eq:defizetaeta} and \eqref{eq:hardyremainderagaintransformedlimit}, we obtain
  \begin{align}
    \label{eq:boglem51aux1}
    \begin{split}
      \ce_{\zeta,\eta}[v] + \int_0^\infty dr\, r^{2\zeta} |v(r)|^2q(r)
      & \geq \ci_{\zeta,\eta}[v] + \int_0^\infty dr\, r^{2\zeta} |v(r)|^2q(r) \\
      & = \ce_\zeta[v], \quad v\in L^2(\R_+,r^{2\zeta}dr).
    \end{split}
  \end{align}
  We now consider an approximation of $\ce_{\zeta,\eta}$. For $\epsilon>0$, let
  \begin{align}
    q^\epsilon(z) := q(z)\wedge \frac1\epsilon
  \end{align}
  and $\tilde p_\zeta^{(\alpha,\epsilon)}$ be the Schr\"odinger perturbation of $p_\zeta^{(\alpha)}$ by $q^\epsilon>0$, i.e.,
  \begin{align}
    p_0^{(\alpha,\epsilon)}(t,r,s) & := p_\zeta^{(\alpha)}(t,r,s), \\
    p_n^{(\alpha,\epsilon)}(t,r,s) & := \int_0^t d\tau \int_0^\infty dz\, z^{2\zeta} p_{n-1}^{(\alpha,\epsilon)}(\tau,r,z)q^\epsilon(z)p_\zeta^{(\alpha)}(t-\tau,z,s), \quad n\in\N, \\
    \tilde p_\zeta^{(\alpha,\epsilon)}(t,r,s) & := \sum_{n\geq0} p_n^{(\alpha,\epsilon)}(t,r,s).
  \end{align}
  Of course, we regard these kernels with the reference (speed) measure $r^{2\zeta}dr$ on $\R_+$.
  By induction,
  \begin{align}
    0 \leq p_n^{(\alpha,\epsilon)}(t,r,s) \leq p_\zeta^{(\alpha)}(t,r,s) \cdot \left(\frac{t}{\epsilon}\right)^n \cdot \frac{1}{n!}, \quad n\in\N_0.
  \end{align}
  Thus,
  \begin{align}
    \begin{split}
      0 \leq p_\zeta^{(\alpha,*)}(t,r,s) & := \tilde p_\zeta^{(\alpha,\epsilon)}(t,r,s) - p_\zeta^{(\alpha)}(t,r,s) - p_1^{(\alpha,\epsilon)}(t,r,s) \\
      & \ \leq p_\zeta^{(\alpha)}(t,r,s) \left(\me{t/\epsilon}-1-\frac{t}{\epsilon}\right)
        = p_\zeta^{(\alpha)}(t,r,s) \cdot \mathcal{O}(t^2) \quad \text{as}\ t\to0
    \end{split}
  \end{align}
  for all $r,s>0$.
  Recalling the normalization \eqref{eq:normalizedalpha}, we observe
  \begin{align}
    \frac1t |\langle v,p^{(\alpha,*)}(t,\cdot,\cdot)v\rangle_{L^2(\R_+,r^{2\zeta}dr)}|
    \lesssim t \|v\|_{L^2(\R_+,r^{2\zeta}dr)}^2 \to0 \quad \text{as}\ t\to0.
  \end{align}
  Next, for $v\in L^2(\R_+,r^{2\zeta}dr)$, we
  have
  \begin{align}
    \footnotesize
    \begin{split}
      \frac1t\langle v,p_1^{(\alpha,\epsilon)}(t,\cdot,\cdot)v\rangle
      & = \frac1t\int_0^\infty dr\int_0^\infty ds \int_0^t d\tau \int_0^\infty dz\, (rsz)^{2\zeta} p_\zeta^{(\alpha)}(\tau,r,z) q^\epsilon(z) p_\zeta^{(\alpha)}(t-\tau,z,s) v(s)v(r) \\
      & = \frac1t \int_0^\infty dz \int_0^t d\tau\, z^{2\zeta} (p_\zeta^{(\alpha)}(\tau,\cdot,\cdot)v)(z) q^\epsilon(z) (p_\zeta^{(\alpha)}(t-\tau,\cdot,\cdot)v)(z).
    \end{split}
  \end{align}
  Since $\{p_\zeta^{(\alpha)}(t,\cdot,\cdot)\}_{t\geq0}$ is a strongly continuous contraction semigroup on $L^2(\R_+,r^{2\zeta}dr)$, the above integral converges to $\int_0^\infty dz\, z^{2\zeta}q^\epsilon(z) v(z)^2$ as $t\to0$. Thus, we obtain that $\{\tilde p_{\zeta}^{(\alpha,\epsilon)}(t,\cdot,\cdot)\}_{t\geq0}$ forms a symmetric strongly continuous semigroup on $L^2(\R_+,r^{2\zeta}dr)$, too. Moreover, each $\tilde p_\zeta^{(\alpha,\epsilon)}(t,\cdot,\cdot)$ is contractive because it is smaller than $p_\zeta^{(\alpha)}(t,\cdot,\cdot)$ on non-negative functions.
  
  Next, note that
  \begin{align*}
    \ce_{\zeta}^{(\epsilon)}(t)[v] := \frac1t\langle v,v- \tilde p_\zeta^{(\alpha,\epsilon)}(t,\cdot,\cdot)v\rangle_{L^2(\R_+,r^{2\zeta}dr)}
    \to \ce_{\zeta}^{(\epsilon)}[v] := \sup_{s>0} \ce_{\zeta}^{(\epsilon)}(s)[v]
    \ \text{as} \ t\to0.
  \end{align*}
  Therefore,
  \begin{align}
    \ce_\zeta^{(\epsilon)}[v]
    = \ce_\zeta[v] - \int_0^\infty dz\, z^{2\zeta} q^\epsilon(z)v(z)^2,
    \quad v\in L^2(\R_+,r^{2\zeta}dr).
  \end{align}
  Since
  $p_{\zeta,\eta}^{(\alpha)}(t,\cdot,\cdot)v \geq \tilde p_\zeta^{(\alpha,\epsilon)}(t,\cdot,\cdot)v$
  for $0\leq v\in L^2(\R_+,r^{2\zeta}dr)$ and all $t>0$,
  we have
  $\ce_{\zeta,\eta}(t)[v]\leq \ce_\zeta^{(\epsilon)}(t)[v]$ for $v\in L^2(\R_+,r^{2\zeta}dr)$
  and all $t>0$. Therefore,
  \begin{align}
    \ce_{\zeta,\eta}[v] + \int_0^\infty dz\, z^{2\zeta}q(z) v(z)^2 \leq \ce_\zeta[v].
  \end{align}
  Combining this inequality with \eqref{eq:boglem51aux1} shows that for all $0\leq v\in L^2(\R_+,r^{2\zeta}dr)$,
  \begin{align}
    \label{eq:boglem51aux2}
    \ce_{\zeta,\eta}[v] + \int_0^\infty dz\, z^{2\zeta}q(z)v(z)^2
    = \ce_\zeta[v]
    = \ci_{\zeta,\eta}[v] + \int_0^\infty dz\, z^{2\zeta}q(z)v(z)^2.
  \end{align}
  For general signed $v\in\D(\ce_\zeta)$, we decompose $v=v_+-v_-$ with $v_+=v\one_{f\geq0}$ and $v_-=v\one_{f\leq0}$ both belonging to $\D(\ce_\zeta)$. Since $\D(\ce_\zeta)$ and $\D(\ce_{\zeta,\eta})$ are linear spaces, we obtain $\D(\ce_\zeta)\subseteq\D(\ce_{\zeta,\eta})$. 
  The extension of \eqref{eq:boglem51aux2} to $v$ with variable sign follows repeating the previous steps using the decomposition of $v$ into its positive and negative parts.
\end{proof}

\subsubsection{Form cores of subordinated Bessel operators}
\label{ss:besseldomainalpha}

Recall that $C_c^\infty(\R^d)$ is dense in $H^{\alpha/2}(\R^d)$ for all real $\alpha$, see, e.g., Bahouri, Chemin, and Danchin \cite[Theorems~1.61]{Bahourietal2011}. Moreover, $C_c^\infty(\R^d\setminus\{0\})$ is dense in $H^{\alpha/2}(\R^d)$ for all $\alpha\leq d$ and $C_c^\infty(\R^d\setminus\{0\})$ is dense in $\dot H^{\alpha/2}(\R^d)$ for all $\alpha<d$; see, e.g., \cite[Theorem~1.70]{Bahourietal2011}.

\smallskip
The goal of this section is to show, in Lemma~\ref{besseldomainalpha}, that $C_c^\infty(\R_+)$ is a form core for $\ci_{\zeta,\eta}$, i.e., functions in $\D(\ci_{\zeta,\eta})$ can be approximated by $C_c^\infty(\R_+)$ functions in the norm $\sqrt{\ci_{\zeta,\eta}[\cdot]} + \|\cdot\|_{L^2(\R+, r^{2\zeta} , dr)}$ for all $\zeta \in (-1/2, \infty)$, $\alpha \in (0, 2\wedge(2\zeta + 1))$, and $\eta \in (0, (2\zeta + 1 - \alpha)/2]$.

In the following, we write $h_\beta(r):=r^{-\beta}$\index{$h_\beta(r)$} for $\beta\in\R$ and $r>0$. Recall \eqref{eq:defh}, i.e., 
$$
h(r) = h_\eta(r).
$$
Let $\chi\in C_c^\infty(\R)$ be a symmetric, non-negative bump function with $\supp\chi\subseteq(-1/2,1/2)$, and such that $\int_\R\chi(r)\,dr=1$. For $0<\epsilon<1$, define $\rho_\epsilon(r):=\epsilon^{-1}\chi(|r-1|/\epsilon)$ so that $\supp \rho_\epsilon\in(1-\epsilon/2,1+\epsilon/2)$. Thus, $\rho_\epsilon(r)$ is an approximation of the identity centered at $r=1$.
For $\epsilon>0$, we define approximations of $f\in L_\loc^1(\R_+)$ by
\begin{align}
  \label{eq:defapproximation}
  f_\epsilon(r) := \int_0^\infty ds\, \rho_\epsilon(s)f(r/s), \quad r>0.
\end{align}
Note that $f_\epsilon\in C^\infty(\R_+)$. Indeed, 
the smoothness follows since we can estimate the derivatives of the integrand in
\begin{align}
  f_\epsilon(r) = \int_0^\infty ds\, r\, s^{-2} \rho_\epsilon\left(\frac rs\right) f(s).
\end{align}
We now show some elementary properties (bounds and behavior at infinity) of such approximations.
\begin{lemma}
  \label{approximation}
  Let $\zeta\in(-1/2,\infty)$, $f\in L_\loc^1(\R_+)$, $\epsilon\in(0,1)$, and $f_\epsilon$ be defined by \eqref{eq:defapproximation}.
  \begin{enumerate}
  \item If $\beta\in\R$ and $h_\beta f\in L^2(\R_+,r^{2\zeta}\,dr)$ then
  \begin{align}
    \label{e.ba}
    \|h_\beta f_\epsilon\|_{L^2(\R_+,r^{2\zeta}dr)}
    \leq \left[(1+\epsilon/2)^{\zeta+1/2-\beta}\vee(1-\epsilon/2)^{\zeta+1/2-\beta}\right] \|h_\beta f\|_{L^2(\R_+,r^{2\zeta}dr)}.
  \end{align}
  
  \item If $f\in L^2(\R_+,r^{2\zeta}dr)$ then $\lim_{r\to \infty}f_\epsilon(r)=0$ and $\lim\limits_{\epsilon\to0} \|f-f_\epsilon\|_{L^2(\R_+,r^{2\zeta}dr)} = 0$.
  \end{enumerate}
\end{lemma}

\begin{proof}
  (1) The $L^2$-bound \eqref{e.ba} follows from Jensen's inequality since
  \begin{align}
    \begin{split}
      \|h_\beta f_\epsilon\|_{L^2(\R_+,r^{2\zeta}dr)}^2
      & = \int_{\R_+}dr\, r^{2\zeta} h_\beta(r)^2 \left|\int_{\R_+}dz\, \rho_\epsilon(z) f(r/z)\right|^2 \\
      & \leq \int_{\R_+}dr\, r^{2\zeta} h_\beta(r)^2 \int_{\R_+}dz\, \rho_\epsilon(z)|f(r/z)|^2 \\
      & = \int_{\R_+}dr\, r^{2\zeta} h_\beta(r)^2 |f(r)|^2 \int_{\R_+}dz\, \rho_\epsilon(z)z^{2\zeta+1-2\beta} \\
      & \leq \|h_\beta f\|_{L^2(\R_+,r^{2\zeta}dr)}^2 \left[(1+\epsilon/2)^{2\zeta+1-2\beta}\vee (1-\epsilon/2)^{\zeta+1/2-\beta}\right].
    \end{split}
  \end{align}
  
  (2) To show that $f_\epsilon(r)$ vanishes as $r\to \infty$, let $p := 2\zeta+2>1$. 
  Since $\supp(\rho_\epsilon) \subseteq [1/2, 3/2]$, by the Cauchy--Schwarz inequality and the change of variables $z' = r/z$,
  \begin{align}
    \begin{split}
      |f_\epsilon(r)|^2
      & = \left|\int_{\R_+} dz\, \rho_\epsilon(z)\left(\frac{z}{r}\right)^{p/2} \left(\frac{r}{z}\right)^{p/2} \, f(r/z) \right|^2 \\
      & \leq \left(\int_{\R_+} dz\, \rho_\epsilon(z)^2\left(\frac{z}{r}\right)^p\right) \cdot \left(\int_{\R_+}\frac{dz}{z}\, z\cdot\left(\frac{r}{z}\right)^p|f(r/z)|^2\right) \\
      & \lesssim_{p,\epsilon} r^{1-p} \int_{\R_+} dz\, z^{2\zeta} |f(z)|^2\to 0 \quad \mbox{as}\quad r\to\infty.
    \end{split}
  \end{align}
  We next prove the $L^2$-approximation. For 
  $f\in L^2(\R_+,r^{2\zeta}dr)$ and $\phi\in C_c^\infty(\R_+)$, 
  $$
  \|f-f_\epsilon\|_{L^2(\R_+,r^{2\zeta}dr)}\le 
  \|f-\phi\|_{L^2(\R_+,r^{2\zeta}dr)}+\|\phi-\phi_\epsilon\|_{L^2(\R_+,r^{2\zeta}dr)}+
  \|(\phi-f)_\epsilon\|_{L^2(\R_+,r^{2\zeta}dr)}.
  $$
  Since $r^{2\zeta}dr$ is a Radon measure on $\R_+$, $C_c^\infty(\R_+)$ is dense in $L^2(\R_+,r^{2\zeta}dr)$; see, e.g., \cite[pp.~189--190]{Schilling2017}. Thus, by \eqref{e.ba} with $\beta=0$, it suffices to consider $f\in C_c^\infty(\R_+)$, and then, by Jensen's inequality,
  \begin{align}
    \begin{split}
      \|f-f_\epsilon\|_{L^2(\R_+,r^{2\zeta}dr)}^2
      & = \int_{\R_+}dr\, r^{2\zeta}\left|\int_{\R_+}dz\, \rho_\epsilon(z)[f(r)-f(r/z)]\right|^2 \\
      & \leq \int_{\R_+}dr\, r^{2\zeta}\int_{\R_+}dz\, \rho_\epsilon(z)|f(r)-f(r/z)|^2.\nonumber
    \end{split}
  \end{align}
  The integrand is zero unless $r\in (1-\epsilon/2,1+\epsilon/2)\cdot \supp f$ and $|z-1|\leq\epsilon/2$, in which case $|r-r/z|<\epsilon r$ and $|f(r)-f(r/z)|\leq \delta$ for any given $\delta>0$, provided $\epsilon$ is small enough.
  Therefore, 
  \begin{align}
    \begin{split}
      & \|f-f_\epsilon\|_{L^2(\R_+,r^{2\zeta}dr)}^2
      \le\delta^2 \, |(1-\epsilon/2,1+\epsilon/2)\cdot \supp f|^{2\zeta+1}/(2\zeta+1).\nonumber
    \end{split}
  \end{align}
  This concludes the proof.
\end{proof}

We next estimate $\ci_{\zeta,\eta}[h_{\eta}f_\epsilon]+ \|h_{\eta} f_\epsilon\|_{L^2(\R_+,r^{2\zeta}\,dr)}^2$ uniformly in $\epsilon$.

\begin{lemma}
  \label{boglem52}
  Let $\zeta\in(-1/2,\infty)$, $\alpha\in(0,2\wedge(2\zeta+1))$, $\eta\in[0,(2\zeta+1-\alpha)/2]$. Then, for all $\epsilon\in(0,1)$, $f\in L_\loc^1(\R_+)$, and $f_\epsilon$ as in \eqref{eq:defapproximation},
  \begin{align}
    \label{eq:boglem52}
    \begin{split}
      & \ci_{\zeta,\eta}[h f_\epsilon] + \|h f_\epsilon\|_{L^2(\R_+,r^{2\zeta}\,dr)}^2
      \lesssim_{\zeta,\alpha,\eta,\rho} \ci_{\zeta,\eta}[h f] + \|h f\|_{L^2(\R_+,r^{2\zeta}\,dr)}^2.
    \end{split}
  \end{align}
\end{lemma}

\begin{proof}
  By Lemma~\ref{approximation}, the bound holds for $\|h f_\epsilon\|_{L^2(\R_+,r^{2\zeta}\,dr)}$. To estimate $\ci_{\zeta,\eta}[h f_\epsilon]$, we apply Jensen's inequality and change variables $r' = r/z$, $s'=s/z$:
  \begin{align}
    \begin{split}
      & 2\ci_{\zeta,\eta}[h f_\epsilon] \\
      & \quad = \int_0^\infty dr\, r^{2\zeta}\int_0^\infty ds\, s^{2\zeta}\, \nu_{\zeta}(r,s) \left|\int_0^\infty dz\, \rho_\epsilon(z)\, (f(\tfrac rz)-f(\tfrac sz))\right|^2 h(r) h(s) \\
      & \quad \leq \int\limits_{(\R_+)^3}dr\,ds\,dz\, (rs)^{2\zeta-\eta}\, \nu_\zeta(r,s)\rho_\epsilon(z)\, |f(\tfrac rz)-f(\tfrac sz)|^2 \\
      & \quad = \int\limits_{(\R_+)^2}dr\,ds\, (rs)^{2\zeta-\eta} \nu_\zeta(r,s) |f(r)-f(s)|^2 \int_{\R_+}dz\, z^{2(\zeta-\eta)+1-\alpha} \rho_\epsilon(z).
    \end{split}
  \end{align}
  Here we used $\nu_\zeta(rz,sz)=z^{-2\zeta-1-\alpha}\nu_\zeta(r,s)$. Since
  \begin{align}
    \label{eq:boglem52aux}
    \int_{\R_+}dz\, z^{2(\zeta-\eta)+1-\alpha} \rho_\epsilon(z)
    \sim_{\zeta,\alpha,\eta,\rho} 1,
    \quad \epsilon\in(0,1),
  \end{align}
  the proof is concluded.
\end{proof}

To show that $C_c^\infty(\R_+)$ is a form core, we will use two bounds in the proof of the following lemma (see \eqref{e.BeBe}--\eqref{e.BeBe2}), which is interesting in its own right. For $R>0$, we write $B_R:=(0,R)$ and $B_R^c:=[R,\infty)$.

\begin{lemma}
  \label{boglem53}
  Let $\zeta\in(-1/2,\infty)$, $\alpha\in(0,2\wedge(2\zeta+1))$, $\eta\in[0,(2\zeta+1-\alpha)/2]$. Let $\chi\in C^\infty(\R_+)$ with $\one_{(1,\infty)}\leq\chi\leq\one_{(1/2,\infty)}$. Then there is $c_{\zeta,\alpha,\eta}>0$ such that for all $f\in C_c^\infty(\R_+)$ and $\epsilon\in(0,1)$,
  \begin{align}
    \label{eq:boglem53}
    \ci_{\zeta,\eta}[h f\, \chi(\cdot/\epsilon)]
    \leq 2\ci_{\zeta,\eta}[h f] + c_{\zeta,\alpha,\eta} \left[(\|\chi'\|_\infty + \|\chi\|_\infty) \|f\|_\infty + \|\chi\|_\infty \|f'\|_\infty\right]^2.
  \end{align}
\end{lemma}

\begin{proof}
  Denote $f^{(\epsilon)}(r):=f(r)\chi(r/\epsilon)$ and $A(r,s):=|f^{(\epsilon)}(r)-f^{(\epsilon)}(s)|^2\nu_\zeta(r,s)(rs)^{2\zeta-\eta}$. Since $f^{(\epsilon)}(r)=f(r)$ for $r\in B_\epsilon^c$,
  \begin{align}
    \iint_{B_\epsilon^c \times B_\epsilon^c} A(r,s)\,dr\,ds
    \leq 2\ci_{\zeta,\eta}[h f].
  \end{align}
  By the mean value theorem,
  \begin{align}
    |f^{(\epsilon)}(r)-f^{(\epsilon)}(s)|
    \leq |r-s| \left(\|f'\|_\infty \|\chi\|_\infty + \epsilon^{-1}\|\chi'\|_\infty\|f\|_\infty\right).
  \end{align}
  This estimate and scaling yield
  \begin{align}
    \label{e.BeBe}
    \begin{split}
      & \iint_{B_{2\epsilon}\times B_{2\epsilon}} A(r,s)dr\,ds \\
      & \quad \leq \epsilon^{-2} \left(\|f'\|_\infty \|\chi\|_\infty + \|\chi'\|_\infty\|f\|_\infty\right)^2 \iint_{B_{2\epsilon}\times B_{2\epsilon}} |r-s|^2\nu_\zeta(r,s)(rs)^{2\zeta-\eta}\,dr\,ds \\
      & \quad = c_{\zeta,\alpha,\eta}\, \epsilon^{2\zeta-2\eta+1-\alpha} \left(\|f'\|_\infty \|\chi\|_\infty + \|\chi'\|_\infty\|f\|_\infty\right)^2 \\
      & \quad \leq c_{\zeta,\alpha,\eta} \left(\|f'\|_\infty \|\chi\|_\infty + \|\chi'\|_\infty\|f\|_\infty\right)^2.
    \end{split}
  \end{align}
   Here $2\zeta-2\eta+1-\alpha\ge 0$ and $c_{\zeta,\alpha,\eta}\lesssim\int_0^1 dr \int_0^1 ds\, (rs)^{2\zeta-\eta}(r+s)^{-2\zeta}|r-s|^{1-\alpha}<\infty$ by elementary computations using $\eta\leq(2\zeta+1-\alpha)/2$ and $\zeta>-1/2$. See also \eqref{eq:defnuell1}.
   
  Finally, by \eqref{eq:defnuell1} and the triangle inequality,
  \begin{align}
    \label{e.BeBe2}
    \begin{split}
      \iint_{B_\epsilon\times B_{2\epsilon}^c} A(r,s)\,dr\,ds
      & \leq 4\|f\|_\infty^2 \iint_{B_\epsilon\times B_{2\epsilon}^c} \nu_\zeta(r,s) (rs)^{2\zeta-\eta}\,dr\,ds \\
      & = c_{\zeta,\alpha,\eta}\, \epsilon^{2\zeta-2\eta+1-\alpha} \|f\|_\infty^2 \\
      & \leq c_{\zeta,\alpha,\eta}\, \|f\|_\infty^2.
    \end{split}
  \end{align} 
  This concludes the proof.
\end{proof}

In the following, we use truncation arguments similar to those of Schilling and Uemura \cite{SchillingUemura2012}.
These involve, for $\epsilon>0$, the function
\begin{align}
  \label{eq:sutruncation}
  T_\epsilon(x) := \left(-\frac1\epsilon\right) \vee \left(x- (-\epsilon)\vee x \wedge \epsilon\right) \wedge \frac1\epsilon, \quad x\in\R.
\end{align}

\begin{figure}[H]
    \centering
    \includegraphics[width=0.5\linewidth]{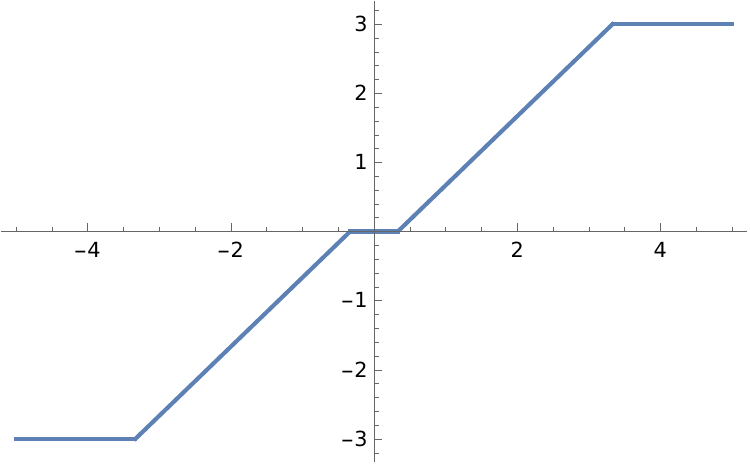}
    \caption{Plot of $T_{1/3}(x)$.}
    \label{fig:schillinguemuraaux}
\end{figure}

\begin{lemma}[{\cite[Lemma~2.3]{SchillingUemura2012}}]
  \label{sutruncation}
  For $\epsilon>0$, the truncation $T_\epsilon(x)$ in \eqref{eq:sutruncation} is a normal contraction, i.e.,
  \begin{align*}
    |T_\epsilon(x)| \leq |x|
    \quad \text{and} \quad|T_\epsilon(x)-T_\epsilon(y)| \leq |x-y|,
    \quad x,y\in\R.
  \end{align*}
  Moreover, for any $x\in\R$, we have $\lim_{\epsilon\to0}T_\epsilon(x)=x$.
\end{lemma}

We now show that $C_c^\infty(\R_+)$ is dense in $\D(\ce_{\zeta})$.

\begin{lemma}
  \label{besseldomainalpha}
  Let $\zeta\in(-1/2,\infty)$ and $\alpha\in(0,2\wedge(2\zeta+1))$. Then $C_c^\infty(\R_+)$ is dense in $\D(\ce_\zeta)$ with respect to $\sqrt{\ce_\zeta[\cdot]}+\|\cdot\|_{L^2(\R_+,r^{2\zeta}dr)}$.
\end{lemma}

For $\alpha=2$, Lemma~\ref{besseldomainalpha} is proved, e.g., in \cite[Theorem~4.22]{Bruneauetal2011} or \cite[Theorem~8.4]{DerezinskiGeorgescu2021}. To ease the notation in the following, define $\F$ to be the closure of $C_c^\infty(\R_+)$ with respect to the norm defined by
\begin{align}
  \sqrt{\ce_\zeta^{(1)}[f]} := \sqrt{\ce_\zeta[f]}+\|f\|_{L^2(\R_+,r^{2\zeta}dr)}, \quad f \in \D(\ce_\zeta).
\end{align}

\begin{proof}[Proof of Lemma~\ref{besseldomainalpha}]
  The proof proceeds analogously to that of Schilling and Uemura \cite[Theorem~2.4]{SchillingUemura2012}. 
  Note that Schilling and Uemura consider Dirichlet forms with some translation symmetry\footnote{This is encoded in the notion of shift-bounded kernels \cite[Definition~2.1]{SchillingUemura2012}.}, which is absent in the present situation.
  To show that every element from $\D(\ce_\zeta)$ can be approximated by a sequence of functions in $C_c^\infty(\R_+)$ with respect to $\ce_\zeta^{(1)}$, let $f\in\D(\ce_\zeta)$. Without loss of generality, let $f$ be real-valued. We perform the following proof for real and imaginary part of $f$ separately, using the linearity of the statement. We define the mollification $\cj_\epsilon[f]:=f_\epsilon$ via \eqref{eq:defapproximation} and note that Schilling and Uemura use the usual convolution instead.
  As observed after \eqref{eq:defapproximation} and in Lemma~\ref{approximation}, $f_\epsilon\in C^\infty(\R_+)$ vanishes at infinity. Thus, $T_\epsilon(f_\epsilon)$ has compact support so $w_\epsilon:=\cj_\epsilon[T_\epsilon(\cj_\epsilon[f])]\in C_c^\infty(\R_+)\subseteq\F$. Moreover, again by Lemma~\ref{approximation}, we have $\lim_{\epsilon\to0}w_\epsilon=f$ in $L^2(\R_+,r^{2\zeta}dr)$. Thus, by Lemma~\ref{boglem52} with $\eta=0$,
  \begin{align*}
    \ce_\zeta[w_\epsilon] \lesssim \ce_\zeta[T_\epsilon(\cj_\epsilon[f])] + \|T_\epsilon(\cj_\epsilon[f])\|_{L^2(\R_+,r^{2\zeta}dr)}^2.
  \end{align*}
  Since $T_\epsilon$ is a normal contraction by Lemma~\ref{sutruncation}, we have
  \begin{align*}
    \ce_\zeta[T_\epsilon(\cj_\epsilon[f])] \leq \ce_\zeta[\cj_\epsilon[f]]
  \end{align*}
  and
  \begin{align*}
    \|T_\epsilon(\cj_\epsilon[f])\|_{L^2(\R_+,r^{2\zeta}dr)}
    \leq \|\cj_\epsilon[f]\|_{L^2(\R_+,r^{2\zeta}dr)}
    \leq \|f\|_{L^2(\R_+,r^{2\zeta}dr)}.
  \end{align*}
  A further application of Lemma~\ref{boglem52} shows
  \begin{align*}
    \ce_\zeta[w_\epsilon]
    \lesssim \ce_\zeta[f] + \|f\|_{L^2(\R_+,r^{2\zeta}dr)}^2,
  \end{align*}
  i.e., the family $\{\ce_\zeta[w_\epsilon]\}_{\epsilon>0}$, hence $\{\ce_\zeta^{(1)}[w_\epsilon]\}_{\epsilon>0}$, is uniformly bounded. Therefore, we can use the Banach--Saks theorem to deduce that for a subsequence $\{\epsilon(n)\}_{n\in\N}$ with $\lim_{n\to\infty}\epsilon(n)=0$ the Cesàro means
  \begin{align*}
    \frac1n \sum_{k=1}^n w_{\epsilon(k)} \in C_c^\infty(\R_+)
  \end{align*}
  converge to a function $\tilde f\in\F$ with respect to $\sqrt{\ce_\zeta^{(1)}}$ and, in particular, in $L^2(\R_+,r^{2\zeta}\,dr)$. On the other hand, we know that $w_\epsilon$, hence any subsequence and any convex combination thereof, converges to $f\in L^2(\R_+,r^{2\zeta}dr)$ as $\epsilon\to0$. Since $L^2$-limits are unique, we conclude $f=\tilde f\in \F$, so $f$ can be approximated by functions in $C_c^\infty(\R_+)$.
\end{proof}

\subsubsection{Proof of equality of $\ci_{\zeta,\eta}$ and $\ce_{\zeta,\eta}$ for $\eta>0$}

We are now ready to show that $C_c^\infty(\R_+)$ is a form core of $\D(\ci_{\zeta,\eta})$.
\begin{proposition}
  \label{smoothdensehardy}
  Let $\zeta\in(-1/2,\infty)$, $\alpha\in(0,2\wedge(2\zeta+1))$, $\eta\in[0,(2\zeta+1-\alpha)/2]$.
  Then $C_c^\infty(\R_+)$ is dense in $\D(\ci_{\zeta,\eta})$.
\end{proposition}

\begin{proof}
  For $\eta=0$, we proved the claim in Lemma~\ref{besseldomainalpha}. Thus, let $\eta>0$.
  When $\eta\in(0,(2\zeta+1-\alpha)/2)$, the claim follows from Hardy's inequality since the Hardy potential can be absorbed by $\ce_{\zeta}$ in this case, i.e.,
  $$
  \left(1-\frac{\Psi_\zeta(\eta)}{\Psi_\zeta((2\zeta+1-\alpha)/2)}\right)\ce_\zeta[u]
  \leq \ci_{\zeta,\eta}[u]
  \leq \ce_{\zeta}[u].
  $$
  Thus, the form norms generated by $\ci_{\zeta,\eta}$ and $\ce_\zeta$ are equivalent and the claim follows from the density of $C_c^\infty(\R_+)$ in $\D(\ce_\zeta)$.

  We now give an argument covering all $\eta\in[0,(2\zeta+1-\alpha)/2]$.
  Observe that $2\zeta-\eta>-1$ and $2(\zeta-\eta)>-1$, and define the form
  \begin{align}
    \begin{split}
      \cq_\eta[f] & := \ci_{\zeta,\eta}[h_{\eta}f],\quad f\in L^2(\R_+,r^{2\zeta}h_{\eta}(r)^2\,dr), \\
      \D(\cq_\eta) & := \frac{1}{h_{\eta}} \D(\ci_{\zeta,\eta}) := \{f/h_{\eta}:\,f\in\D(\ci_{\zeta,\eta})\}.
    \end{split}
  \end{align}
  Thus, $\D(\ci_{\zeta,\eta})=h_{\eta}\D(\cq_\eta)$. Now let $f\in\D(\cq_{\eta})$. By Lemma~\ref{boglem52}, the sequence $f_n(r):=\int_0^\infty dz\, \rho_{1/n}(z)f(r/z)$, $r>0$, is bounded in the norm $\sqrt{\cq_{\eta}[\cdot]} + \|\cdot\|_{L^2(\R_+,r^{2\zeta}h_{\eta}(r)^2 dr)}$.
  Since $\D(\cq_\eta)$ is a Hilbert space, by the Banach--Saks theorem, there exists a subsequence $\{f_{n_k}\}$ and $g\in\D(\cq_{\eta})$ such that
  \begin{align}
    \frac 1k \sum_{m=1}^k f_{n_m} \to g.
  \end{align}
  Since $f_n\to f$ in $L^2(\R_+,r^{2\zeta}h_{\eta}(r)^2\,dr)$, we have $g=f$. Thus, $C^\infty(\R_+)\cap\D(\cq_{\eta})$ is dense in $\D(\cq_{\eta})$.

  Now fix $f\in C^\infty(\R_+)\cap\D(\cq_\eta)$. Let $\chi_1\in C_c^\infty(\R_+)$ be such that $\one_{B_1}\leq\chi_1\leq\one_{B_2}$. Denote $f_1=f\chi_1$ and $f_2=f-f_1$. Then, $f=f_1+f_2$. By \eqref{e.BeBe}--\eqref{e.BeBe2} in the proof of Lemma~\ref{boglem53}, $C_c^\infty(\R_+)\subseteq\D(\cq_\eta)$. Hence $f_1\in\D(\cq_\eta)$ and so $f_2\in\D(\cq_\eta)$. Since $g=h_{\eta}f_2\in\D(\ci_{\zeta,\eta})$ and
  \begin{align}
    \int_{\R_+}|g(r)|^2 q(r)\,r^{2\zeta}\,dr
    \leq \Psi_\zeta(\eta)\int_{B_1^c}|f(r)|^2 h_{\eta}(r)^2 r^{2\zeta}\,dr
    < \infty,
  \end{align}
  we get $g\in\D(\ce_\zeta)$ by the ground state representation~\eqref{eq:hardyremainderagaintransformedlimit}. By Lemma~\ref{besseldomainalpha}, $\ce_\zeta$ with $\alpha<2\zeta+1$ is a regular Dirichlet form with core $C_c^\infty(\R_+)$.
  Thus, there is a sequence $(g_n)_{n\in\N}\subseteq C_c^\infty(B_{1/2}^c)$, which converges to $g$ in $\sqrt{\ce_\zeta[\cdot]}+\|\cdot\|_{L^2(\R_+,r^{2\zeta}dr)}$ as $n\to\infty$. Hence, the sequence $w_n:=g_n/h_{\eta}\in C_c^\infty(\R_+\setminus B_{1/2})$ converges to $f_2$ in the norm $\sqrt{\cq_\eta[\cdot]}+\|\cdot\|_{L^2(\R_+,r^{2\zeta}h_{\eta}(r)^2\,dr)}$ as $n\to\infty$. Finally, a sequence $f_1+w_n\in C_c^\infty(\R_+)$ converges in $\D(\cq_\eta)$ to $f$ as $n\to\infty$. This means that $C_c^\infty(\R_+)$ is dense in $\D(\cq_\eta)$ with respect to the norm $\sqrt{\cq_\eta[\cdot]} + \|\cdot\|_{L^2(\R_+,r^{2\zeta}h_{\eta}(r)^2\,dr)}$.
\end{proof}

\subsubsection{Proof of Theorem~\ref{bogthm51}}

For $\eta=0$, we proved the claim in Lemma~\ref{besseldomainalpha}. Thus, let $\eta>0$.
By Theorem~\ref{propertiesschrodheatkernel}(6) and \cite[Lemmas~1.3.2 and~1.3.4 and Theorem~1.3.1]{Fukushimaetal2011} we obtain that $(\ce_{\zeta,\eta},\D(\ce_{\zeta,\eta}))$ is closed, i.e., that $(\D(\ce_{\zeta,\eta}),\sqrt{\ce_{\zeta,\eta}[\cdot]}+\|\cdot\|_{L^2(\R_+,r^{2\zeta}dr)})$ is a Hilbert space.
Next, we consider $\ci_{\zeta,\eta}$ defined by \eqref{eq:defizetaeta} and claim that also $(\D(\ci_{\zeta,\eta}),\sqrt{\ci_{\zeta,\eta}[\cdot]}+\|\cdot\|_{L^2(\R_+,r^{2\zeta}dr)})$ is a Hilbert space. To that end let $f_n$ be a Cauchy sequence in $(\D(\ci_{\zeta,\eta}),\sqrt{\ci_{\zeta,\eta}[\cdot]}+\|\cdot\|_{L^2(\R_+,r^{2\zeta}dr)})$. Since $f_n$ is bounded in this norm and also clearly a Cauchy sequence in $L^2(\R_+,r^{2\zeta}dr)$, it converges to some $f$ in $L^2(\R_+,r^{2\zeta}dr)$. Moreover, there is a subsequence $f_{n_k}$ which converges pointwise (almost everywhere) to this $f$. By Fatou's lemma, we obtain
\begin{align}
  \ci_{\zeta,\eta}[f] \leq \liminf_{k\to\infty}\ci_{\zeta,\eta}[f_{n_k}] <\infty,
\end{align}
so we have $f\in\D(\ci_{\zeta,\eta})$. Since $f_n$ is Cauchy, we obtain, by Fatou's lemma,
\begin{align}
  \lim_{n\to\infty}\ci_{\zeta,\eta}[f-f_n]
  \leq \lim_{n\to\infty}\liminf_{k\to\infty} \ci_{\zeta,\eta}[f_{n_k}-f_n] = 0,
\end{align}
so $f_n$ also converges to $f$ in $(\D(\ci_{\zeta,\eta}),\sqrt{\ci_{\zeta,\eta}[\cdot]} + \|\cdot\|_{L^2(\R_+,r^{2\zeta}dr)})$. This shows that $(\D(\ci_{\zeta,\eta}),\sqrt{\ci_{\zeta,\eta}[\cdot]}+\|\cdot\|_{L^2(\R_+,r^{2\zeta}dr)})$ is a Hilbert space.
By Proposition~\ref{smoothdensehardy}, $C_c^\infty(\R_+)$ is dense in $\D(\ci_{\zeta,\eta})$ with the norm $\sqrt{\ci_{\zeta,\eta}[\cdot]}+\|\cdot\|_{L^2(\R_+,r^{2\zeta}dr)}$.
By Lemma~\ref{boglem51}, $\ce_{\zeta,\eta}$ and $\ci_{\zeta,\eta}$ coincide on $C_c^\infty(\R_+)$. Since $\D(\ce_{\zeta,\eta})\subseteq\D(\ci_{\zeta,\eta})$ (by Lemma~\ref{boglem51}) and both spaces are complete, they are equal. \qed

\subsection{Equality of $\ci_{\zeta,\eta}$ and $\ce_{\zeta,\eta}$ for $\eta\in(-\alpha,0)$}
\label{s:connectiongeneratorsjw}

In this subsection, we prove Theorem~\ref{relationhardyformheatkernel}(2) for $\alpha<2$ and $\eta\in(-\alpha,0)$.
Let $\zeta\in(-1/2,\infty)$, $\alpha\in(0,2\wedge(2\zeta+1))$, $\eta\in(-\alpha,0)$.

\subsubsection{Ground state representation for $\eta<0$}
We start with a ground state representation for the form sum $\ce_\zeta[u]-\Psi_\zeta(\eta)\int_0^\infty \frac{|u(r)|^2}{r^\alpha}\,r^{2\zeta}dr$ with $\eta\in(-\alpha,0)$. Note that, due to our construction of ground states for $\eta<0$ in \cite{Bogdanetal2024} using compensation, the results of \cite{BogdanMerz2024}, where $\eta>0$ is only considered and where no compensation is necessary, are not applicable in the present situation.

\begin{lemma}
  \label{gsreprrepulsive}
  Let $\zeta\in(-1/2,\infty)$, $\alpha\in(0,2\wedge(2\zeta+1))$, and $\eta\in(-\alpha,0)$. Then, for $u\in L^2(\R_+,r^{2\zeta}dr)$,
  \begin{align}
    \ce_\zeta[u]
    & = \ci_{\zeta,\eta}[u] + \Psi_\zeta(\eta)\int_{\R_+} \frac{|u(r)|^2}{r^\alpha}\,r^{2\zeta}\,dr.
  \end{align}
\end{lemma}

The proof relies on the following result from the integral analysis in \cite{Bogdanetal2024}.

\begin{lemma}[{\cite[Lemma~3.14]{Bogdanetal2024}}]
  \label{jaklem23}
  Let $\zeta\in(-1/2,\infty)$, $\alpha\in(0,2)$, $\eta\in(-\alpha,0)$. Then
  \begin{align}
    \label{eq:jaklem232}
    \begin{split}
      \int_0^\infty ds\, s^{2\zeta} p_\zeta^{(\alpha)}(t,r,s) s^{-\eta}
      = r^{-\eta} - \Psi_\zeta(\eta)\int_0^t d\tau \int_0^\infty ds\, s^{2\zeta} p_\zeta^{(\alpha)}(\tau,r,s) s^{-\eta-\alpha}.
    \end{split}
  \end{align}
\end{lemma}

\begin{proof}[Proof of Lemma~\ref{gsreprrepulsive}]
  It suffices to show $\ce_\zeta[hu]-\ci_{\zeta,\eta}[hu]=\Psi_\zeta(\eta)\int_{\R_+}dr\, r^{2\zeta-2\eta-\alpha}|u(r)|^2$ for $hu\in L^2(\R_+,r^{2\zeta}dr)$. By linearity and polarization, we may and will assume that $u\geq0$. By $\lim_{t\to0}t^{-1}p_\zeta^{(\alpha)}(t,r,s)=\nu_\zeta(r,s)$, $t^{-1}p_\zeta^{(\alpha)}(t,r,s)\lesssim\nu_\zeta(r,s)$, dominated convergence, symmetry,
  and integrating \eqref{eq:jaklem232} against $r^{2\zeta-\eta}|u(r)|^2dr$, i.e.,
  \begin{align}
    \begin{split}
      & \iint_{\R_+\times\R_+} dr\,ds\, (rs)^{2\zeta-\eta}p_\zeta^{(\alpha)}(t,r,s)|u(r)|^2 \\
      & \quad = \int_0^\infty dr\, r^{2\zeta-\eta}|u(r)|^2\left[r^{-\eta} - \Psi_\zeta(\eta)\int_0^t d\tau \int_0^\infty ds\, s^{2\zeta} p_\zeta^{(\alpha)}(\tau,r,s) s^{-\alpha-\eta}\right],
    \end{split}
  \end{align}
  we get
  \begin{align}
    \begin{split}
      \ce_\zeta[hu] - \ci_{\zeta,\eta}[hu]
      = \Psi_\zeta(\eta) \lim_{t\to0}\frac1t \iint_{\R_+\times\R_+}dr\,ds\, (rs)^{2\zeta-\eta}|u(r)|^2 s^{-\alpha} \int_0^t d\tau\, p_\zeta^{(\alpha)}(\tau,r,s).
    \end{split}
  \end{align}
  Thus, it suffices to show
  \begin{align}
    \label{eq:approximationidentityaveraged}
    \lim_{t\to0}\frac1t \iint_{\R_+\times\R_+}dr\,ds\, (rs)^{2\zeta-\eta}|u(r)|^2 s^{-\alpha} \int_0^t d\tau\, p_\zeta^{(\alpha)}(\tau,r,s)
    = \int_0^\infty dr\, r^{2\zeta-2\eta}q(r)|u(r)|^2.
  \end{align}
  Since $p_\zeta^{(\alpha)}(r,s)(rs)^{\zeta}$ is the fundamental solution of the heat equation $(\partial_t+U^{-1}\fl_\zeta U)u=0$ with the unitary $U:L^2(\R_+,dr)\to L^2(\R_+,r^{2\zeta}dr)$ defined by $(Uf)(r)=r^{-\zeta}f(r)$,
  \begin{align}
    \label{eq:approximationidentityheat}
    \lim_{\tau\to0}p_\zeta^{(\alpha)}(\tau,r,s) = (rs)^{-\zeta} \delta(r-s).
  \end{align}
  Thus, \eqref{eq:approximationidentityaveraged} follows from \eqref{eq:approximationidentityheat} and the mean value theorem, once we show that $t\mapsto p_\zeta^{(\alpha)}(t,r,s)$ is continuous for all $r,s>0$ with $r\neq s$. For $\alpha=2$, this follows from the explicit formula \eqref{eq:defpheatalpha2}. For $\alpha\in(0,2)$, subordination and the continuity and uniform boundedness of $t\mapsto \sigma_t^{(\alpha/2)}(\tau)$ (see \cite{Pollard1946} or \cite[Proposition~B.1]{BogdanMerz2024}) imply the continuity of $t\mapsto p_\zeta^{(\alpha)}(t,r,s)$ for all $r,s>0$ with $r\neq s$ by dominated convergence.
\end{proof}

\subsubsection{Proof of equality of $\ci_{\zeta,\eta}$ and $\ce_{\zeta,\eta}$ for $\eta<0$}


\begin{proposition}
  \label{jakprop317}
  Let $\zeta\in(-1/2,\infty)$, $\alpha\in(0,2\wedge(2\zeta+1))$, and $\eta\in(-\alpha,0)$.
  Then $(\ce_{\zeta,\eta},\D(\ce_{\zeta,\eta}))$ is a symmetric regular Dirichlet form on $L^2(\R_+,r^{2\zeta}dr)$ with core $C_c^\infty(\R_+)$.
\end{proposition}

\begin{proof}
  By Lemma~\ref{besseldomainalpha}, $(\ce_\zeta,\D(\ce_\zeta))$ is a symmetric regular Dirichlet form with core $C_c^\infty(\R_+)$. 
  Of course, 
  \begin{align}
    \int_0^\infty \frac{|u(r)|^2}{r^\alpha}\,r^{2\zeta}\,dr < \infty, \quad u\in C_c^\infty(\R_+).
  \end{align}
  In particular, $C_c^\infty(\R_+)\subseteq \D(\ce_\zeta) \cap L^2(\R_+,r^{2\zeta-\alpha}dr) \subseteq\D(\ce_{\zeta,\eta})$.

  It is easy to prove that $(\ce_{\zeta,\eta},\D(\ce_{\zeta,\eta}))$ is a symmetric Dirichlet form on $L^2(\R_+,r^{2\zeta}dr)$. To show that this form is regular, we prove that $C_c^\infty(\R_+)$ is dense in $\D(\ce_{\zeta,\eta})$ with the norm $\sqrt{\ce_{\zeta,\eta}[\cdot]}+\|\cdot\|_{L^2(\R_+,r^{2\zeta}dr)}$. By Hardy's inequality,
  \begin{align}
    \int_0^\infty \frac{|u(r)|^2}{r^\alpha}\,r^{2\zeta}\,dr \lesssim \ce_\zeta[u].
  \end{align}
  Thus, the norms $\sqrt{\ce_{\zeta,\eta}[\cdot]}+\|\cdot\|_{L^2(\R_+,r^{2\zeta}dr)}$ and $\sqrt{\ce_{\zeta}[\cdot]}+\|\cdot\|_{L^2(\R_+,r^{2\zeta}dr)}$ are equivalent. Thus, since $(\ce_\zeta,\D(\ce_\zeta))$ is regular with core $C_c^\infty(\R_+)$, so is $(\ce_{\zeta,\eta},\D(\ce_{\zeta,\eta}))$.
\end{proof}

In the following, we prove two results that are similar to Lemma~4.9 and~4.19 in~\cite{Bogdanetal2024}. We defer their proofs to Appendix~\ref{a:auxiliary}.
Recall the definition of $p_t^{(1,D)}$ in~\eqref{eq:feynmankactransformed}
and let
\begin{align}
  \label{eq:defgattractive}
  G_\eta(t,r,s) := \int_0^t d\tau \int_0^\infty dz\, z^{2\zeta-\alpha-\eta}
  \left(\frac{r+s+z}{s+z}\right)^{2\zeta} p_\zeta^{(\alpha)}(\tau,r,z)
\end{align}\index{$G_\eta(t,r,s)$}
and
\begin{align}
  \tilde G(t,r,s) := \int_0^t d\tau \int_0^{s} dz\, z^{2\zeta-\alpha} \left(\frac{t^{1/\alpha}+r+s}{(t-\tau)^{1/\alpha}+z+s}\right)^{2\zeta} p_\zeta^{(\alpha)}(\tau,r,z).
\end{align}\index{$\tilde G(t,r,s)$}Note that in general $G_\eta(t,r,s)\neq G_\eta(t,s,r)$ and $\tilde G(t,r,s)\neq \tilde G(t,s,r)$. By the scaling \eqref{eq:scalingalpha} of $p_\zeta^{(\alpha)}(\tau,r,z)$,
\begin{align}
  \label{eq:scalinggattractive}
  G_\eta(t,r,s) = t^{-\frac{\eta}{\alpha}}G_\eta(1,r/t^{1/\alpha},s/t^{1/\alpha})
  \quad \text{and} \quad
  \tilde G(t,r,s) = \tilde G(1,r/t^{1/\alpha},s/t^{1/\alpha}).
\end{align}

The following is an analog of \cite[Lemma~4.9]{Bogdanetal2024}.
\begin{lemma}
  \label{glogestimatesattractive}
  Let $\zeta\in(-1/2,\infty)$, $\alpha\in(0,2)$, $\eta\in(-\alpha,2\zeta+1-\alpha)$, and $G_\eta(t,r,s)$ be as in~\eqref{eq:defgattractive}. Then, for all $r,s>0$,
  \begin{align}
    \label{eq:glogestimatesattractive}
    \begin{split}
      G_\eta(1,r,s)
      & \lesssim (r\wedge s)^{-\eta}(rs)^{-\alpha} + s^{-\alpha-\eta} + \log(1+1/r^\alpha)\one_{\eta=0} \\
      & \quad + (1+r)^{-\alpha-\eta}\one_{\eta<0} + r^{-(\alpha+\eta)}(1\wedge r^\alpha)\one_{\eta>0}
    \end{split}
  \end{align}
  and, for $\zeta\in(-1/2,0)$,
  \begin{align}
    \label{eq:glogestimatesattractivenew}
    \begin{split}
      \tilde G(1,r,s) \lesssim \left(\frac{1+r+s}{1+s}\right)^{2\zeta}\log(1+r^{-\alpha}). 
    \end{split}
  \end{align}
\end{lemma}

The following is an analog of \cite[Lemma~4.19]{Bogdanetal2024}.
\begin{lemma}
  \label{jaklem37}
  Let $\zeta\in(-1/2,\infty)$ and $\alpha\in(0,2)$. Then, for all $r,s,t>0$, we have
  \begin{align}
    \label{eq:jaklem37}
    \begin{split}
      p_t^{(1,D)}(r,s)
      & \lesssim p_\zeta^{(\alpha)}(t,r,s) \left[ \left(G_0(t,r,s) + G_0(t,s,r)\right)\one_{\zeta\geq0} \right. \\ & \qquad\qquad\qquad \left. + \left( \frac{t}{(r\wedge s)^{\alpha}} + \tilde G(t,r,s)+\tilde G(t,s,r)\right)\one_{\zeta<0}\right].
    \end{split}
  \end{align}
\end{lemma}

The limit in the following lemma is natural in view of $p_{\zeta,\eta}^{(\alpha)}(t,r,s)\leq p_\zeta^{(\alpha)}(t,r,s)$.

\begin{lemma}
  \label{nuellrepulsive}
  Let $\zeta\in(-1/2,\infty)$, $\alpha\in(0,2\wedge(2\zeta+1))$, and $\eta\in(-\alpha,0)$. Then,
  \begin{align}
    \label{eq:nuellrepulsive}
    \lim_{t\to0}\frac{p_{\zeta,\eta}^{(\alpha)}(t,r,s)}{t}
    = \nu_\zeta(r,s),\quad r,s>0,\ r\neq s.
  \end{align}
\end{lemma}

\begin{proof}
  By Duhamel's formula, the uniform bound \eqref{eq:defnuell4}, i.e., $p_\zeta^{(\alpha)}(t,r,s)/t\lesssim\nu_\zeta(r,s)$, and the definition of $\nu_\zeta(r,s)$ in \eqref{eq:defnuell1}, it suffices to show that
  \begin{align}
    \begin{split}
      \lim_{t\to0}\frac1t \int_0^t d\tau \int_0^\infty dz\, z^{2\zeta} p_{\zeta,\eta}^{(\alpha)}(t-\tau,r,z)z^{-\alpha}p_\zeta^{(\alpha)}(\tau,z,s) = 0.
    \end{split}
  \end{align}
  But this follows from $p_{\zeta,\eta}^{(\alpha)}(t,r,s)\leq p_\zeta^{(\alpha)}(t,r,s)$, Lemma~\ref{jaklem37} and the bounds for $G_0(t,r,s)$ and $\tilde G(t,r,s)$ in Lemma~\ref{glogestimatesattractive} with the scaling \eqref{eq:scalinggattractive}.
\end{proof}

We are now ready to prove the equivalence of $\ce_{\zeta,\eta}$ and $\ci_{\zeta,\eta}$ for $\eta\in(-\alpha,0)$.

\begin{theorem}
  \label{jakprop318}
  Let $\zeta\in[0,\infty)$, $\alpha\in(0,2\wedge(2\zeta+1))$, and $\eta\in(-\alpha,0)$.
  Then, $\D(\ce_{\zeta,\eta})=\D(\ci_{\zeta,\eta})$ and
  \begin{align}
    \ce_{\zeta,\eta}[u] = \ci_{\zeta,\eta}[u],\quad u\in\D(\ce_{\zeta,\eta}).
  \end{align}
\end{theorem}

\begin{proof}
  By Theorem~\ref{propertiesschrodheatkernel}(4),
  \begin{align}
    \begin{split}
      & \langle u,u-p_{\zeta,\eta}^{(\alpha)}(t,\cdot,\cdot)u\rangle_{L^2(\R_+,r^{2\zeta}dr)}
      = \int_{\R_+}dr\, r^{2\zeta}\, \overline{u(r)}\left(u(r) - \int_{\R_+}ds\, s^{2\zeta} p_{\zeta,\eta}^{(\alpha)}(t,r,s) u(s)\right) \\
      & \quad = \int_{\R_+}dr\, r^{2\zeta}\, \overline{u(r)}\left(\frac{u(r)}{h(r)} \int_0^\infty ds\, s^{2\zeta} p_{\zeta,\eta}^{(\alpha)}(t,r,s)h(s) - \int_0^\infty ds\, s^{2\zeta} p_{\zeta,\eta}^{(\alpha)}(t,r,s)h(s)\frac{u(s)}{h(s)}\right) \\
      & \quad = \iint_{\R_+\times\R_+}dr\,ds\, (rs)^{2\zeta-\eta}\, \frac{\overline{u(r)}}{h(r)} p_{\zeta,\eta}^{(\alpha)}(t,r,s)\left(\frac{u(r)}{h(r)} - \frac{u(s)}{h(s)}\right).
    \end{split}
  \end{align}
  Hence, by symmetry,
  \begin{align}
    \langle u,u-p_{\zeta,\eta}^{(\alpha)}(t,\cdot,\cdot)u\rangle_{L^2(\R_+,r^{2\zeta}dr)}
    = \frac12 \iint_{\R_+\times\R_+}dr\,ds\, (rs)^{2\zeta-\eta}\, p_{\zeta,\eta}^{(\alpha)}(t,r,s)\left|\frac{u(r)}{h(r)} - \frac{u(s)}{h(s)}\right|^2.
  \end{align}
  Now, for $u\in\D(\ce_{\zeta,\eta})$, by Fatou's lemma and \eqref{eq:nuellrepulsive}, we have
  \begin{align}
    \begin{split}
      \ce_{\zeta,\eta}[u]
      & = \lim_{t\to0}\frac1t\langle u,(1-p_{\zeta,\eta}^{(\alpha)}(t,\cdot,\cdot))u\rangle_{L^2(\R_+,r^{2\zeta}dr)} \\
      & = \lim_{t\to0}\frac12\iint_{\R_+\times\R_+}\frac{p_{\zeta,\eta}^{(\alpha)}(t,r,s)}{t} \left|\frac{u(r)}{h(r)} - \frac{u(s)}{h(s)}\right|^2 (rs)^{2\zeta-\eta}\,dr\,ds \\
      & \geq \frac12\iint\limits_{\R_+\times\R_+}\liminf_{t\to0}\frac{p_{\zeta,\eta}^{(\alpha)}(t,r,s)}{t} \left|\frac{u(r)}{h(r)} - \frac{u(s)}{h(s)}\right|^2 (rs)^{2\zeta-\eta}\,dr\,ds \\
      & = \frac12\iint\limits_{\R_+\times\R_+}dr\,ds\, (rs)^{2\zeta-\eta}\nu_\zeta(r,s) \left|\frac{u(r)}{h(r)} - \frac{u(s)}{h(s)}\right|^2
        = \ci_{\zeta,\eta}[u].
    \end{split}
  \end{align}
  This shows that $\D(\ce_{\zeta,\eta})\subseteq\D(\ci_{\zeta,\eta})$. To prove the reversed inclusion, let $u\in\D(\ci_{\zeta,\eta})$. Then, by $p_{\zeta,\eta}^{(\alpha)}(t,r,s)\leq p_\zeta^{(\alpha)}(t,r,s)\lesssim t\nu_\zeta(r,s)$ for all $r,s,t>0$, \eqref{eq:defnuell4}, by the dominated convergence and \eqref{eq:nuellrepulsive},
  \begin{align}
    \begin{split}
      \ce_{\zeta,\eta}[u]
      & = \lim_{t\to0}\frac1t\langle u,(1-p_{\zeta,\eta}^{(\alpha)}(t,\cdot,\cdot))u\rangle_{L^2(\R_+,r^{2\zeta}dr)} \\
      & = \lim_{t\to0}\frac12 \iint_{\R_+\times\R_+}\frac{p_{\zeta,\eta}^{(\alpha)}(t,r,s)}{t} \left|\frac{u(r)}{h(r)} - \frac{u(s)}{h(s)}\right|^2 h(r)h(s)\, (rs)^{2\zeta}\,dr\,ds \\
      & = \frac12 \iint_{\R_+\times\R_+}dr\,ds\, (rs)^{2\zeta-\eta}\nu_\zeta(r,s) \left|\frac{u(r)}{h(r)} - \frac{u(s)}{h(s)}\right|^2
      = \ci_{\zeta,\eta}[u].
    \end{split}
  \end{align}
  This shows $u\in\D(\ce_{\zeta,\eta})$ and hence $\D(\ci_{\zeta,\eta})\subseteq\D(\ce_{\zeta,\eta})$
  and $\ce_{\zeta,\eta}[u]=\ci_{\zeta,\eta}[u]$ for all $u\in\D(\ce_{\zeta,\eta})$.
\end{proof}

\section{Proof of Theorem~\ref{mainresult}}
\label{s:proof}

The estimates \eqref{eq:mainresultfreealpha}--\eqref{eq:mainresultfree2} (concerning $\eta=0$) follow from Theorem~\ref{relationhardyformheatkernel}(2) and Proposition~\ref{heatkernelalpha1subordinatedboundsfinal} with $\zeta=(d_\ell-1)/2$.

Formula~\eqref{eq:mainresultalpha} (concerning $\alpha<2$ and $\eta\neq0$) follows from Theorems~\ref{relationhardyformheatkernel}(2) and~\ref{mainresultgen} with $\zeta=(d_\ell-1)/2$.

Finally, \eqref{eq:mainresult2} (concerning $\alpha=2$ and $\eta\neq0$) is proved by Metafune, Negro, and Spina \cite[Theorem~4.12, Proposition~4.14]{Metafuneetal2018} using that $\cl_{\Phi_{d_\ell}^{(2)}(\eta),\ell}$ in $L^2(\R_+,r^{d_\ell-1}dr)$ is unitarily equivalent to $\fl_{(d_\ell-1)/2-\eta}$ in $L^2(\R_+,r^{d_\ell-1-2\eta}dr)$. Thus,
\begin{align}
  \exp\left(-t\cl_{\Phi_{d_\ell}^{(2)}(\eta),\ell}\right)(r,s) = (rs)^{-\eta} \exp\left(-t\fl_{(d_\ell-1)/2-\eta}\right)(r,s) \quad \text{for $\alpha=2$}.
\end{align}
This, together with 
\begin{align}
  \label{eq:heatkernellzeta}
  \me{-t\fl_\zeta}(r,s) = p_\zeta^{(2)}(t,r,s), \quad t,r,s>0,
\end{align}
\eqref{eq:easybounds2} and elementary estimates, gives \eqref{eq:mainresult2}.
\qed

\appendix

\section{Fourier--Bessel transform}

For the following results, see \cite[Appendix~A]{BogdanMerz2024}.
For $d\in\N$ and $\ell\in L_d$, the Fourier--Bessel transform $\F_{d_\ell}$,
\begin{align}
  \label{eq:deffourierbessel}
  v\mapsto (\F_{d_\ell} v)(k) := i^{-\ell}\int_0^\infty dr\, \sqrt{kr}J_{\frac{d_\ell-2}{2}}(kr) v(r),
  \quad k\in\R_+
\end{align}\index{$\F_\ell$}is initially defined for functions $v\in L^1(\R_+,dr)$. Note also that this transform is well-defined because $(d_\ell-1)/2>-1/2$. The same arguments used to extend the Fourier transform on $L^1(\R^d)$ to $L^2(\R^d)$ allow one to extend the Fourier--Bessel transform to a unitary operator on $L^2(\R_+,dr)$. By abuse of notation, we denote this extension by $\F_{d_\ell}$, too, and observe the symmetry $\F_{d_\ell}=\F_{d_\ell}^*$, which is apparent from \eqref{eq:deffourierbessel}. Using the unitary operator
\begin{align}
  \label{eq:defdoob}
  L^2(\R_+,r^{d_\ell-1}dr) \ni u \mapsto (Uu)(r)=r^{(d_\ell-1)/2}u(r) \in L^2(\R_+,dr)
\end{align}
and the plane wave expansion, we obtain the intertwining relation
\begin{align}
  \label{eq:fourierbesselintertwine2}
  \left([u]_{\ell,m}\right)^\wedge(\xi) = \left(U^*[\F_{d_\ell} Uu]_{\ell,m}\right)(\xi),
  \quad \xi\in\R^d,\, u\in L^2(\R_+,r^{d_\ell-1}dr).
\end{align}

The following lemma says that the Fourier--Bessel transform diagonalizes $F(|D|)$, when restricted to a fixed $V_{\ell}$ or $V_{\ell,m}$. In the following, we write $F(|D|)^{1/2}$ as short for $|F(|D|)|^{1/2}\sgn(F(|D|))$.

\begin{lemma}[{\cite[Lemma~4.1]{BogdanMerz2024}}]
  \label{fourierbessel}
  Let $d\in\N$, $\ell,\ell'\in L_d$, $m\in M_\ell$, and $m'\in M_{\ell'}$. Let $F\in L_\loc^1(\R_+)$ and $\cq(F(|D|)):=\{g\in L^2(\R^d):\, \int_{\R^d}|F(|\xi|)||\hat g(\xi)|^2\,d\xi<\infty\}$. Then, for $[u]_{\ell,m},\,[v]_{\ell',m'}\in \cq(F(|D|))$,
  \begin{align}
    \label{eq:fourierbessel1}
    \begin{split}
      & \langle |F(|D|)|^{1/2} [v]_{\ell',m'},F(|D|)^{1/2} [u]_{\ell,m}\rangle_{L^2(\R^d)} \\
      & \quad = \delta_{\ell,\ell'}\delta_{m,m'}\langle \F_{d_\ell} Uv,F \cdot \F_{d_\ell} Uu\rangle_{L^2(\R_+,dr)}
    \end{split}
  \end{align}
  with the unitary operator $U:L^2(\R_+,r^{d_\ell-1}dr)\to L^2(\R_+,dr)$ defined in \eqref{eq:defdoob}.
\end{lemma}

\section{Proofs of auxiliary statements}
\label{a:auxiliary}

To prove Lemma~\ref{glogestimatesattractive}, we recall the following lemma from \cite{Bogdanetal2024}.

\begin{lemma}[{\cite[Lemma~3.1]{Bogdanetal2024}}]
 \label{integralana1}
 Let $\zeta\in(-1/2,\infty)$, $\alpha\in(0,2]$ and $\delta \in (0,2\zeta+1)$. Then,
 \begin{align}
    \label{eq:hbetagammataurintegratedtransformed}
    \begin{split}
       \int_0^t d\tau\, \int_0^\infty ds\, s^{2\zeta} p_\zeta^{(\alpha)}(\tau,r,s) s^{-\delta} \; \sim \;
        \begin{cases}
          t (r^\alpha \vee t)^{-\delta/\alpha} & \quad \text{for}\ \delta<\alpha \\
          \log(1+\frac{t}{r^\alpha}) & \quad \text{for}\ \delta=\alpha \\
          r^{-\delta}  (t\wedge r^\alpha)  & \quad \text{for}\ \delta>\alpha
        \end{cases}.
    \end{split}
  \end{align}
\end{lemma}

\subsection{Proof of Lemma~\ref{glogestimatesattractive}}

We first prove \eqref{eq:glogestimatesattractive}.
Using
\begin{align}
  \label{eq:deffaux}
  \frac{r+z}{s+z} \one_{\{r>s\vee z\}} + \one_{\{r<s\vee z\}}
  \sim \frac{r+s+z}{s+z},
\end{align}
we estimate
\begin{align}
  \label{eq:defgattractivesim}
  G_\eta(t,r,s) \sim \int_0^t d\tau \int_0^\infty dz\, z^{2\zeta} \cdot z^{-\alpha} \cdot z^{-\eta} \left[\left(\frac{r+z}{s+z}\right)^{2\zeta}\one_{\{r>s\vee z\}} + \one_{\{r<s\vee z\}}\right] p_\zeta^{(\alpha)}(\tau,r,z).
\end{align}
We now treat the first summand in the brackets on the right-hand side of \eqref{eq:defgattractivesim}. By \eqref{eq:heatkernelalpha1weightedsubordinatedboundsfinal},
\begin{align}
  (r+z)^{2\zeta} p_\zeta^{(\alpha)}(\tau,r,z)
  \lesssim \frac{\tau}{|r-z|^{1+\alpha} + \tau^{(1+\alpha)/\alpha}}.
\end{align}
Thus,
\begin{align}
  \label{eq:glogestimatesaux1attractive}
  \footnotesize
  \begin{split}
    & \int_0^1 d\tau \int_0^\infty dz\, z^{2\zeta-\alpha-\eta}\left(\frac{r+z}{s+z}\right)^{2\zeta}\one_{\{r>s\vee z\}} p_\zeta^{(\alpha)}(\tau,r,z) \\
    & \quad \lesssim \int_0^1 d\tau \int_0^\infty dz\, z^{2\zeta-\alpha-\eta} \cdot \frac{\tau}{|r-z|^{1+\alpha}+\tau^{(1+\alpha)/\alpha}} \cdot \frac{\one_{\{r>s\}}\one_{\{r>z\}}}{(s+z)^{2\zeta}} \\
    & \quad \lesssim \int_0^1 d\tau \int_0^{s/2}dz\, \left(\frac{z}{s}\right)^{2\zeta} z^{-\alpha-\eta} \frac{\tau\cdot \one_{\{r>s\}}\one_{\{r>z\}}}{r^{1+\alpha}+\tau^{(1+\alpha)/\alpha}}
      + \int_0^1 d\tau \int_{s/2}^\infty dz\, z^{-\alpha-\eta} \frac{\tau\cdot \one_{\{r>s\}}\one_{\{r>z\}}}{|r-z|^{1+\alpha}+\tau^{(1+\alpha)/\alpha}}.
  \end{split}
\end{align}
The first summand on the right of \eqref{eq:glogestimatesaux1attractive} is estimated by
\begin{align}
  \begin{split}
    \int_0^1 d\tau \int_0^{s/2}dz\, \left(\frac{z}{s}\right)^{2\zeta} z^{-\alpha-\eta} \frac{\tau\cdot \one_{\{r>s\}}\one_{\{r>z\}}}{r^{1+\alpha}+\tau^{(1+\alpha)/\alpha}}
    \lesssim \frac{1}{(rs)^{\alpha}(r\wedge s)^{\eta}}.
  \end{split}
\end{align}
We now come to the second summand in \eqref{eq:glogestimatesaux1attractive} and estimate
\begin{align}
  \footnotesize
  \begin{split}
    & \int_0^1 d\tau \int_{s/2}^r dz\, z^{-\alpha-\eta} \frac{\tau}{|r-z|^{1+\alpha}+\tau^{(1+\alpha)/\alpha}}
      \lesssim s^{-\alpha-\eta}\int_0^1 d\tau \int_{0}^r dz \frac{\tau}{\tau^{(1+\alpha)/\alpha}+(r-z)^{1+\alpha}} \\
    & \quad \lesssim s^{-\alpha-\eta} \int_0^1 d\tau \int_0^\infty dz\, \frac{\tau}{\tau^{(1+\alpha)/\alpha} + z^{1+\alpha}}
      \lesssim s^{-\alpha-\eta}.
  \end{split}
\end{align}
Thus, the left-hand side of \eqref{eq:glogestimatesaux1attractive} is bounded by
\begin{align}
  \int_0^1 d\tau \int_0^\infty dz\, z^{2\zeta-\alpha-\eta} \left(\frac{r+z}{s+z}\right)^{2\zeta} \one_{\{r>s\vee z\}} p_\zeta^{(\alpha)}(\tau,r,z)
  \lesssim (r\wedge s)^{-\eta} (rs)^{-\alpha} + s^{-\alpha-\eta}.
\end{align}
We now come to the second summand in \eqref{eq:defgattractivesim}. By Lemma~\ref{integralana1},
\begin{align}
  \begin{split}
    & \int_0^1 d\tau \int_0^\infty dz\, z^{2\zeta} \cdot z^{-\alpha-\eta} \cdot \one_{\{r<s\vee z\}}\cdot p_\zeta^{(\alpha)}(\tau,r,z) \\
    & \quad \lesssim \log(1+1/r^\alpha)\one_{\eta=0} + (1+r^\alpha)^{-1-\eta/\alpha} \one_{\eta<0} + r^{-(\alpha+\eta)}(1\wedge r^\alpha)\one_{\eta>0}.
  \end{split}
\end{align}
This concludes the proof of \eqref{eq:glogestimatesattractive}.

We now prove \eqref{eq:glogestimatesattractivenew}. Using Lemma~\ref{integralana1}, we obtain
\begin{align}
  \begin{split}
    \int_0^1 d\tau \int_0^{s} dz\, z^{2\zeta-\alpha}\frac{(1+r+s)^{2\zeta}}{((1-\tau)^{1/\alpha}+z+s)^{2\zeta}}p_\zeta^{(\alpha)}(\tau,r,z)
    \lesssim \left(\frac{1+r+s}{1+s}\right)^{2\zeta}\log(1+r^{-\alpha}),
  \end{split}
\end{align}
as desired. This concludes the proof. \qed

\subsection{Proof of Lemma~\ref{jaklem37}}

By scaling, it suffices to consider $t=1$. Consider first $\zeta\geq0$. According to the 3G inequality \cite[Theorem~3.1]{BogdanMerz2025}
\begin{align}
  \label{eq:3gheatalpha2}
  & p_\zeta^{(\alpha)}(t,r,z)\cdot p_\zeta^{(\alpha)}(\tau,z,s) \\
  & \ \lesssim p_\zeta^{(\alpha)}(t+\tau,r,s) \cdot \left[ \left(\frac{r+s+z}{s+z}\right)^{2\zeta} p_\zeta^{(\alpha)}(t,r,z) + \left(\frac{r+s+z}{r+z}\right)^{2\zeta} p_\zeta^{(\alpha)}(\tau,z,s)\right] \notag
\end{align}
where $r,s,t,\tau>0$ and $\zeta\geq0$, and to the definition of $G_\eta(t,r,s)$, we get
\begin{align}
  \begin{split}
    & p_1^{(1,D)}(r,s) = \int_0^1 d\tau \int_0^\infty dz\, z^{2\zeta} p_{\zeta}^{(\alpha)}(1-\tau,r,z) q(z) p_\zeta^{(\alpha)}(\tau,z,s) \\
    & \quad \lesssim p_\zeta^{(\alpha)}(1,r,s) \int_0^1 d\tau \int_0^\infty dz\, z^{2\zeta-\alpha} \\
    & \qquad \times \left[\left(\frac{r+s+z}{s+z}\right)^{2\zeta} \cdot p_\zeta^{(\alpha)}(\tau,r,z) 
      + \left(\frac{r+s+z}{r+z}\right)^{2\zeta} \cdot p_\zeta^{(\alpha)}(\tau,z,s)\right] \\
    & \quad = p_\zeta^{(\alpha)}(1,r,s) \left[G_0(1,r,s) + G_0(1,s,r)\right].
  \end{split}
\end{align}
Now let $\zeta\in(-1/2,0)$. Here, we use another 3G inequality, contained in \cite[Lemma~2.5]{Bogdanetal2024}, namely
\begin{align}
  \label{eq:3gheatalphanew}
  & p_\zeta^{(\alpha)}(t,r,z)\,p_\zeta^{(\alpha)}(\tau,z,s) \\
  & \quad \lesssim p_\zeta^{(\alpha)}(t+\tau,r,s) \left((t+\tau)^{1/\alpha}+r+s\right)^{2\zeta} \left[\frac{p_\zeta^{(\alpha)}(t,r,z)}{(\tau^{1/\alpha}+z+s)^{2\zeta}} + \frac{p_\zeta^{(\alpha)}(\tau,z,s)}{(t^{1/\alpha}+r+z)^{2\zeta}} \right], \notag
\end{align}
which holds for all $\zeta\in(-1/2,\infty)$.
Thus, using the Chapman--Kolmogorov equations for small $z$ and the 3G inequality for large $z$ and the definition of $\tilde G(t,r,s)$, we obtain
\begin{align}
  \begin{split}
    p_1^{(1,D)}(r,s)
    & \lesssim \int_0^1 d\tau \int_{r\wedge s}^\infty dz\, z^{2\zeta-\alpha}p_\zeta^{(\alpha)}(1-\tau,r,z)p_\zeta^{(\alpha)}(\tau,z,s) \\
    & \quad + p_\zeta^{(\alpha)}(1,r,s) (1+r+s)^{2\zeta} \\
    & \qquad \times \int_0^1 d\tau \int_0^\infty dz\, z^{2\zeta-\alpha} \left[\frac{p_\zeta^{(\alpha)}(1-\tau,r,z)}{(\tau^{1/\alpha}+z+s)^{2\zeta}}\one_{z<s} + \frac{p_\zeta^{(\alpha)}(\tau,s,z)}{((1-\tau)^{1/\alpha}+z+r)^{2\zeta}}\one_{z<r}\right] \\
    & = p_\zeta^{(\alpha)}(1,r,s)\left((r\wedge s)^{-\alpha} + \tilde G(1,r,s) + \tilde G(1,s,r)\right).
  \end{split}
\end{align}
This concludes the proof. \qed

\begin{remark}
  In \cite{Bogdanetal2024}, we already noted that the two 3G inequalities in \eqref{eq:3gheatalpha2} and~\eqref{eq:3gheatalphanew} do not imply each other. This manifests itself again in the proof of Lemma~\ref{glogestimatesattractive}.
  Proving \eqref{eq:jaklem37} for $\zeta\geq0$ using the 3G inequality in~\eqref{eq:3gheatalphanew} seems to be difficult due to the complicated dependencies on the time-variable and so we use \eqref{eq:3gheatalpha2} instead. On the other hand, when $\zeta\in(-1/2,0)$, it is convenient to use \eqref{eq:3gheatalphanew}.
  %
\end{remark}

\printindex


\newcommand{\etalchar}[1]{$^{#1}$}
\def\cprime{$'$}

\end{document}